\documentclass[12pt]{article}
\usepackage{amsmath}
\usepackage{helvet}
\usepackage{amssymb}
\usepackage[margin=1 in]{geometry}
\usepackage{amsthm}
\usepackage{lipsum}
\usepackage{chngcntr}
\PassOptionsToPackage{square, numbers}{natbib}
\usepackage[square,numbers]{natbib}

\newtheorem{assumption}{Assumption}
\newtheorem{lemma}{Lemma}[section]

\newtheorem{theorem}{Theorem}[section]
\newtheorem{prop}{Proposition}[section]
\theoremstyle{definition}
\newtheorem{remark}{Remark}
\usepackage{chngcntr}
\PassOptionsToPackage{square, numbers}{natbib}
\usepackage[square,numbers]{natbib}
\usepackage{relsize}
\usepackage{titlesec}
\usepackage{mathtools}
\theoremstyle{plain}
\newtheorem*{theorem*}{Theorem}
\usepackage{hyperref}
\usepackage{xcolor}
\usepackage[T1]{fontenc}
\usepackage[utf8]{inputenc}
\usepackage{authblk}

\numberwithin{equation}{section}

\begin{document}
 \title{\textbf{Stochastic homogenization of deterministic control problems}}
 \date{\today}
 \author{Alexander Van-Brunt  \thanks{alexvb@kurims.kyoto-u.ac.jp}}
 \affil{ \textit{Research Institute for Mathematical Sciences},\\
\textit{Kyoto University}}
\maketitle
\begin{abstract}
In this paper we study homogenization of a class of control problems in a stationary and ergodic random environment. This problem has been mostly studied in the calculus of variations setting in connection to the homogenization of the Hamilton-Jacobi equations. We extend the result to the control problems with fairly general state dynamics and macroscopically inhomogeneous Lagrangians. Moreover, our approach proves homogenization under weaker  growth assumptions on the Lagrangian even in the well-studied calculus of variations setting. 
 \end{abstract}

\tableofcontents

 %%%%%%%%%%%%%%%%%%%%%%%%%%%%%%%%%%%%%%%%%%%%%%%%%%%%%%%%%%%%%%%%%%%%%%%%%%%%%%%%%%%%%%%%%%%%%%%%%%%%%%%%%%%%%%%%%%%%%%%
\section{Introduction}
%%%%%%%%%%%%%%%%%%%%%%%%%%%%%%%%%%%%%%%%%%%%%%%%%%%%%%%%%%%%%%%%%%%%%%%%%%%%%%%%%%%%%%%%%%%%%%%%%%%%%%%%%%%%%%%%%%%%%%%%

We study homogenization of deterministic optimal control problem in random, ergodic, environments with state dynamics beyond the calculus of variations and macroscopically inhomogeneous Lagrangians. The control problem we will be considering is as follows (we refer the reader to \cite{fleming2006controlled} and references therein for general theory on the optimal control problems). Let $(\Omega, \mathcal{F}, P)$ be a probability space and 
\begin{equation}
L\colon \mathbb{R}\times\mathbb{R}^d\times\mathbb{R}^d\times\mathbb{R}^d\times\Omega \to \mathbb{R}
\end{equation}
be a random field, which we call the \emph{Lagrangian}. The Lagrangian $L$ will be assumed to be stationary and ergodic with respect to the translation in the third argument, see Section $1.4$ below for the precise formulation. 
%Henceforth, unless explicitly stated, a.s. shall be referring $P$ almost surely. 
Fix an $x \in \mathbb{R}^{d}$, $t\in (0, T]$ for some terminal time $T>0$.  For a given $\omega \in \Omega$, an $\epsilon > 0$ a terminal cost function $\psi(\cdot)$, define the cost functional
 \begin{equation}
 J_{\epsilon} (t,x,u, \omega):=\int_{t}^{T} L \left(  s,x(s), \frac{x(s)}{\epsilon}, u(s), \omega  \right) ds + \psi(x(T))
\label{cost-intro}
 \end{equation}
over $ u \in \mathcal{U}:=L^{\infty} \left(  [0,T] \rightarrow \mathbb{R}^{d} \right)$, where $x(\cdot)$ is defined by the so-called \emph{state dynamics}
 \begin{equation}
\begin{cases}
\frac{dx}{ds}(s) =f(x(s),u(s)) &\textrm{ for }s\in (t,T],\\
x(t)=x
\end{cases}
\label{statedynamics}
\end{equation}
for a suitable $f\colon \mathbb{R}^d\times\mathbb{R}^d\to\mathbb{R}^d$.
The control problem studies the minimal cost which is called the \emph{value function}:
 \begin{equation}
 V_{\epsilon}(t,x,\omega) := \inf_{u \in \mathcal{U}} J_{\epsilon} (t,x,u,\omega).
 \end{equation}
For a fixed $\epsilon>0$, this value function is random depending on $\omega$. But if we let $\epsilon\to 0$, then quite often some ``averaging'' takes place and $V_{\epsilon}(t,x,\omega)$ converges to a non-random quantity. Proving this convergence, together with the identification of the limit, is called the problem of \emph{homogenization}. 
Thus our objective is to prove, under appropriate assumptions, the existence of an \emph{effective Lagrangian} $\widetilde{L}$ such that
\begin{equation}
V_{\epsilon}(t,x) \rightarrow \widetilde{V} (t,x)\textrm{ as }\epsilon\to 0
\label{hom-intro}
\end{equation}
uniformly on compact sets, where $\widetilde{V}$ is the value function corresponding to the cost functional
 \begin{equation}
 \widetilde{J} (t,x,u)=\int_{t}^{T} \widetilde{L}  \left ( s,x(s), u(s) \right ) ds + \psi(x(T))
 \end{equation}
with the same state dynamics \eqref{statedynamics} as before. 

Under suitable coercivity assumptions, this result automatically implies homogenization of the associated Hamilton-Jacobi equation. Indeed if we define the Hamiltonians
\begin{align}
\mathcal{H} \left(t,x,\frac{x(s)}{\epsilon}, p, \omega\right)
&=\sup_{v \in \mathbb{R}^{d}} \left\{ - f(x,v) \cdot p - L\left( t,x,\frac{x}{\epsilon}, v, \omega \right) \right\},\label{intro-H}\\
 \widetilde{\mathcal{H}} (t,x,p)&= \sup_{v \in \mathbb{R}^{d}} \left\{ - f(x,v) \cdot p - \widetilde{L}(t,x,v) \right\},\label{intro-Hhom}
\end{align}
then it is a well known fact that, under our assumptions presented later, the value functions $V_{\epsilon}$ and $\widetilde{V}$ are the unique viscosity solutions to the following Hamilton-Jacobi equations: 

\begin{align}
-\frac{\partial V_{\epsilon} }{\partial t} +\mathcal{H} \left(t,x,\frac{x(s)}{\epsilon}, D_{x} V_{\epsilon}, \omega\right)&=0,\\
-\frac{\partial \widetilde{V}}{\partial t} +\widetilde{\mathcal{H}} (t,x, D_{x} \widetilde{V})&=0\label{homogenHJB}
\end{align}
subject to the terminal conditions $\widetilde{V}(T,x)=V(T,x)=\psi(x)$. Therefore, the convergence \eqref{hom-intro} implies homogenization of the Hamilton-Jacobi equation. 
%\begin{remark}
%The result is restricted to the zero terminal condition since we assumed zero terminal cost in~\eqref{cost-intro}. However, the approach of this paper readily extends to non-zero terminal costs. In fact, it is possible to prove the homogenization of our control problem with an extra terminal constraint $x(T)=z\in\mathbb{R}^d$. This implies the homogenization of the value function $V^\psi_\epsilon$ with the terminal cost $\psi$ through the variational formula 
%\begin{equation}
% V^\psi_\epsilon(t,x,\omega)=\inf_{u\in\mathcal{U}}
%\left\{\psi(x(T))+J_\epsilon(t,x,u,\omega)\right\}. 
%\end{equation}
%\end{remark}

%%%%%%%%%%%%%%%%%%%%%%%%%%%%%%%%%%%%%%%%%%%%%%%%%%%%%%%%%%%%%%%%%%%%%%%%%%%%%%%%%%%%%%%%%%%%%%%%%%%%%%%%%%%%%%%%%%%%%%%%
 \subsection{Background and earlier results}
 %%%%%%%%%%%%%%%%%%%%%%%%%%%%%%%%%%%%%%%%%%%%%%%%%%%%%%%%%%%%%%%%%%%%%%%%%%%%%%%%%%%%%%%%%%%%%%%%%%%%%%%%%%%%%%%%%%%%%%%
 
As mentioned in the introduction, our problem is naturally related to the homogenization of Hamilton-Jacobi equations.  There has been a substantial interest in this topic. The applications of which include large deviations of diffusion process in random environments, front propagation in random media and first passage percolation theory. 
%Under appropriate assumptions, one can show the existence of an effective Hamiltonian. More precisely, if $V_{\epsilon}$ denotes the unique viscosity solution to the Hamilton Jacobi equation
%\begin{equation}
%-\frac{\partial V_{\epsilon} }{\partial t} +\mathcal{H} (t,x,\frac{x(s)}{\epsilon}, D_{x} V_{\epsilon}, \omega)=0
%\end{equation}
%where $H$ is a Hamiltonian satisfying appropriate coercity and convexity conditions, then it can be shown there exists an $\widetilde{H}$ such that $V_{\epsilon} \rightarrow \widetilde{V}$ where $\widetilde{V}$ is the unique viscosity solution to the `effective' Hamilton Jacobi equation
%
%\begin{equation}
%-\frac{\partial \widetilde{V}}{\partial t} +\widetilde{\mathcal{H}} (t,x, D_{x} \widetilde{V},)=0
% \end{equation}
 Such a result has been obtained by several authors. In periodic setting it has been studied via reducing it to a cell problem, as first carried out by~\cite{lions1986homogenization} for temporary homogeneous Hamiltonians. The result was later extended to almost periodic case with general Hamiltonians in~\cite{Ish99} by using the so called \emph{perturbed test function method} intiated in~\cite{Eva89}. In the stochastic case, these approaches are no longer possible since the corrector does not exist in general~\cite{LS03}. This case was first addressed in~\cite{souganidis1999stochastic} and~\cite{Rezakhanlou2000} where the authors applied the subadditive ergodic theorems to the control theoretic interpretation of the solution. It is worth mentioning that in order to find an associated Lagrangian, the Hamiltonian is assumed to be convex. In the non-convex setting, it is known that homogenization does not occur in general~\cite{Zil17,feldman2017homogenization}. Later these approaches are further extended to the second order equations with a vanishing viscosity by~\cite{lions2005homogenization} and~\cite{CPA:CPA20137}. There is much recent progress in the homogenization of Hamilton-Jacobi equations and we refer the reader to the introduction in ~\cite{JST17}. 

In this paper, we are interested in homogenization of control problems with general state dynamics. Such a problem naturally appears in financial and engineering applications. All the above works mainly focus on the Hamilton-Jacobi equations and the associated control problems are assumed to follow the simplest state dynamics $f(x,u)=u$, a problem sometimes called the \emph{calculus of variations}. The control problem is only in the background and hence there is no reason to consider other complicated state dynamics. 
Needless to say, one can use the above results to show the convergence of $V_\epsilon$ to $\widetilde{V}$ provided the Hamiltonian~\eqref{intro-H} satisfies the  appropriate conditions. However, through this procedure, the effective Hamiltonian is not necessarily given in the form of the right hand side of~\eqref{intro-Hhom}. If $f$ is non-linear in $v$, then it leaves a highly nontrivial problem of recovering the effective Lagrangian associated with the original state dynamics. 
Another disadvantage of this procedure is that we will end up with a rather implicit assumptions on $L$ and $f$, saying that the Hamiltonian~\eqref{intro-H} satisfies certain conditions. It may not be hard to find a reasonable sufficient conditions but it may well not be the most natural conditions. 

For these reasons, we develop a direct approach to prove the homogenization of control problems with general state dynamics. This allows us to obtain the effective Lagrangian and also leads to explicit assumptions in terms of $L$ and $f$. 
In addition, our approach has two technical merits. First, even in the calculus of variations setting, our assumptions are seen to be less restrictive than the ones in~\cite{souganidis1999stochastic} and~\cite{Rezakhanlou2000}. Also, it gives a straightforward way to include macroscopic inhomogeneity to the Lagrangian. Both~\cite{souganidis1999stochastic} and~\cite{Rezakhanlou2000} assumed the Hamiltonian (and hence Lagrangian) to have the form $H(x/\epsilon,\omega)$. There is a way to extend the results to the macroscopically inhomogeneous cases by using the viscosity solution framework, as carried out in~\cite{lions1986homogenization,lions2005homogenization}, but our approach gives an alternative elementary way. 

We also hope that the approach of this paper can be extended beyond the scope of current setting, such as vanishing viscosity case. 

%%%%%%%%%%%%%%%%%%%%%%%%%%%%%%%%%%%%%%%%%%%%%%%%%%%%%%%%%%%%%%%%%%%%%%%%%%%%%%%%%%%%%%%%%%%%%%%%%%%%%%%%%%%%%%%%%%%%%%%%
\subsection{Outline of the proof}
%%%%%%%%%%%%%%%%%%%%%%%%%%%%%%%%%%%%%%%%%%%%%%%%%%%%%%%%%%%%%%%%%%%%%%%%%%%%%%%%%%%%%%%%%%%%%%%%%%%%%%%%%%%%%%%%%%%%%%%%

In the general setting, we will make a number of assumptions on the Lagrangian $L$ and state dynamics $f$. In order to motivate them, we shall explain a brief outline of the argument.\\

For simplicity, we shall assume that $\psi(x)=0$ and there exists a Lipschitz continuous optimal control $u^*$. The proof of our result is based on a discretization scheme. Specifically to homogenize the control problem, we wish to use a certain ergodic theorem on the object

\begin{equation}
V_{\epsilon}(t,x,\omega) =\int_{t}^{T} L\left(s ,x(s), \frac{x(s)}{\epsilon}, u^{*}(s), \omega\right) ds .
\end{equation}
However, this is not straightforward. One immediate reason is that, since we have no a priori information on $u$, the object is not stationary in the ``fast variable'' $x(\cdot)/\epsilon$ where we expect the averaging to occur. Another issue is that the stationarity and ergodicity of $L$ is assumed only for fixed $(s,x(s))$.
We may  attempt to cope with this by introducing a discretization. For $\tau>0$ let $t_{0}=t$ and $t_{i}=t_{i-1}+\tau \wedge T$ until the first $i$ such that $t_{i}=T$. This sequence $\{ t_{i} \}_{i=1}^{N}$ defines a partition of $[t,T]$.  Define $\bar{x}(\cdot)$ as the piecewise linear function with $\bar{x}(t_{i})=x(t_i)$ and consider
\begin{equation}
\sum_{i=1}^N \int_{t_i}^{t_{i+1}}L\left(t_{i},x(t_{i}), \frac{\bar{x}(s)}{\epsilon},u^{*}(t_{i}),\omega\right) ds.\label{naive}
\end{equation}
Since $\bar{x}(\cdot)$ is a straight line segment on $[t_i,t_{i+1})$ and the other variables of $L$ are frozen, there is no problem in applying the ergodic theorem on each $[t_i,t_{i+1})$. If there were no fast variable, it is an easy exercise to show that~\eqref{naive} approximates the value function $V_\epsilon$ well. 

% Suppose we have Lipschtiz continuous Lagrangian $L$ and a optimal continuous control $u^{*}$, then for any sequence of intervals $[t_{i},t_{i+1})$ partitioning $[t,T]$  we can write
%
%\begin{equation}
%V(t,x) = \int_{t}^{T} L(s,x(s),u^{*}(s)) ds= \sum_{i} \int_{t_{i}}^{t_{i+1}} L(s,x(s),u^{*}(s)) ds
%\end{equation}
%
%We might realistically expect that this could be approximated as 
%
%\begin{equation} \label{heursticapproximationclassic}
%V(t,x) = \sum_{i} (t_{i+1} -t_{i} ) L(t_{i},x(t_{i}),u^{*}(t_{i}))  + \mathcal{O}(\max_{i} (t_{i+1}-t_{i})) 
%\end{equation}

If however we applied this reasoning to our case of interest we would only conclude
\begin{equation}
V_{\epsilon}(t,x, \omega) = \sum_{i=1}^N
\int_{t_i}^{t_{i+1}} L\left(t_{i},x(t_{i}),\frac{\bar{x}(s)}{\epsilon},u^{*}(t_{i}),\omega\right) ds +\mathcal{O}(1). 
\label{outline-fail}
\end{equation}
Then we see such an approximation would not be useful. The issue is that no matter how small $|x(s)-\bar{x}(s)|$ is, it will be magnified by $\epsilon^{-1}$ and hence we cannot conclude that the values of the Lagrangian are close. This is not counterintuitive since we expect the control to fluctuate more wildly as the trajectory attempts to navigate through the random environment. 

To address this concern, we instead discretize the control problem in a way which makes no attempt to approximate the controls on the microscopic level, but rather approximates the controls on a macroscopic level and implicitly optimizes over the microscopic environment. This corresponds to introducing the following objects:
\begin{enumerate}
\item For some $i < N$ and  $u\in\mathcal{U}$, let $\widehat{x}(\cdot)$ follow the \emph{frozen} dynamics $\frac{d\widehat{x}(s)}{ds}=f(x(t_{i}),u(s))$ for $s\in[t_i,t_{i+1})$ with $\widehat{x}(t_i)=x(t_i)$.
\item Suppose there exists $\widetilde{u}_{i}$ such that
\begin{equation}
f(x(t_{i}),\widetilde{u}_{i})=\frac{x(t_{i+1})-x(t_{i})}{\tau}
\label{tildeu}
\end{equation}
and let
\begin{equation}
\Pi_{t_{i},x(t_{i}),\tau,\widetilde{u}_{i}}=\left\{ u \in \mathcal{U} \: : \:\int_{t_{i}}^{t_{i+1}} f(x(t_{i}),u(s))ds= \tau f(x(t_{i}),\widetilde{u}_{i}) \right\}. \label{outline-Pi}
\end{equation}
\item Let
\begin{equation} \label{Macrocostfunction}
L_{\tau,\epsilon}(t_{i},x(t_{i}),\widetilde{u}_{i},\omega): = \inf_{u \in \Pi_{t_{i},x(t_{i}),\tau,\widetilde{u}_{i}}}  \int_{t_{i}}^{t_{i+1}} L\left(t_{i},x(t_{i}),\frac{x(s)}{\epsilon}, u(s),\omega \right) ds.
\end{equation}
\end{enumerate}
This definition requires some explanation. The $\widetilde{u}_i$ is superfluous since by substituting its defining equation to \eqref{outline-Pi}, we see that the above is simply the minimal cost to get from $x(t_{i})$ to $x(t_{i+1})$. Its role will become clear later. 
In the special case of the calculus of variations, this $\widetilde{u}_i$ is given simply as $\tau^{-1}\int_{t_{i}}^{t_{i+1}} u^{*}(s)ds$, i.e., the average direction of the trajectory over the time $\tau$. %The point of this is that we can at each stage, construct a discrete control problem based on choosing these vectors $\widetilde{u}$. 
In the general case, we will need a convexity assumption on the image of $f(x(t_i),\cdot)$ to guartanee the existence of such an $\widetilde{u}_{i}$. 

Now if we could show that the use of frozen dynamics has little effect to the cost, i.e., 
 \begin{equation}
L_{\tau,\epsilon}(t_{i},x(t_{i}),\widetilde{u}_{i},\omega) \sim \int_{t_{i}}^{t_{i+1}} L\left(s,x^{*}(s), \frac{x^{*}(s)}{\epsilon}, u^{*}(s)\right)ds, 
 \end{equation} 
then we obtain the following improvement of~\eqref{outline-fail}:
\begin{equation}  \label{heuristicapproximation}
 V_{\epsilon}(t,x,\omega)=
 \sum_{i=1}^N L_{\tau,\epsilon}(t_{i},x(t_{i}),\widetilde{u}_{i},\omega) +o(1).
\end{equation}
Thanks to the frozen dynamics, we can show that $L_{\tau,\epsilon}$ is a stationary subadditive process. 
This allows us to use the subadditive ergodic theorem to deduce the almost sure existence of the limit
\begin{equation} \label{convergenceofsubaddergotheorem}
\lim_{\epsilon \rightarrow 0} L_{\tau,\epsilon}(t_{i},x(t_{i}),\widetilde{u}_{i},\omega)=  \widetilde{L}_{\tau}(t_{i},x(t_{i}),\widetilde{u}),
\end{equation}
which is positively homogeneous in $\tau$. 
We can then define an \emph{effective Lagrangian} by 
\begin{equation}
 \widetilde{L}(t_{i},x(t_{i}),\widetilde{u}_{i}) =\frac{\widetilde{L}_{\tau}(t_{i},x(t_{i}),\widetilde{u}_{i})}{\tau}.
\end{equation}
If we know that $\widetilde{L}$ is jointly continuous, combined with the assumed ``regularity'' of $u^*$, we can proceed as
\begin{equation} \label{heurstichomogenization}
%\begin{split}
 V_{\epsilon} (t,x,\omega) 
\sim \sum_{i=1}^N (t_{i+1}-t_i) \widetilde{L}(t_i,x(t_i),u^*(t_i))
\sim \int_{t}^{T} \widetilde{L}(s,x(s),u^*(s))ds
\ge \widetilde{V}(t,x).
%\end{split}
\end{equation}
The proof of the other direction is done by reserving the above procedure. \\

Let us comment on several key points in the above outline. First of all, the use of subadditive ergodic theorem is not new in this context. It is used in~\cite{souganidis1999stochastic,Rezakhanlou2000} and many later works also rely on it. So the additional difficulties in the generality that we present here is the regularity issues: the existence of ``regular'' approximate optimal control, the stability of $\widetilde{L}$ in modification (freezing) of the state dynamics, and the continuity of $\widetilde{L}$. 
For the first issue, we will simply show that there is a nearly step function control that approximates our value function. 
The second and third issues are more subtle and closely related. The difficulty is highlighted by the fact that we will want to prove that the effective Lagrangian $\widetilde{L}$ is continuous via proving a equicontinuity result on $\widetilde{L}_{\tau, \epsilon}$. As we have seen in~\eqref{outline-fail}, the cost functional 
\begin{equation}
\int L\left(s,x(s),\frac{x(s)}{\epsilon},u(s),\omega\right)ds
\end{equation}
is highly unstable in $x(\cdot)$. Thus in order to compare the two costs with different starting points, we cannot use a simple continuity argument, we have to construct controls for which one trajectory essentially traces out the other one exactly and hence acquires no error in the microscopic variable. 
Certainly, this type of argument requires some flexibility to the state dynamics. In addition to this, since we will assume the continuity of $L$ only locally uniform in the control variable, it is also important to introduce a type of \emph{effective compactness} of the control space which ensures that the control does not get too large. This requires a coercivity assumption that roughly states that, as the control variable gets large, the cost grows faster than the available speed. 

%%%%%%%%%%%%%%%%%%%%%%%%%%%%%%%%%%%%%%%%%%%%%%%%%%%%%%%%%%%%%%%%%%%%%%%%%%%%%%%%%%%%%%%%%%%%%%%%%%%%%%%%%%%%%%%%%%%%%%%%
 \subsection{Assumptions and main results}
 %%%%%%%%%%%%%%%%%%%%%%%%%%%%%%%%%%%%%%%%%%%%%%%%%%%%%%%%%%%%%%%%%%%%%%%%%%%%%%%%%%%%%%%%%%%%%%%%%%%%%%%%%%%%%%%%%%%%%%%
 
 We first recall the definition of stationarity and ergodicity. Let $ \{ \varphi_{x} \colon \Omega\to\Omega\}_{x\in \mathbb{R}^d}$ be a group of measure preserving transformations on the probability space $(\Omega, \mathcal{F}, P)$, that is  for any $x,y\in \mathbb{R}^d$, $\varphi_{x+y}=\varphi_{x}\circ\varphi_{y}$ and $P\circ\varphi_{x}^{-1}=P$.  The transformation group $\varphi $ is said to be ergodic if $P(A) \in \{ 0,1 \}$  for all $A \in \mathcal{F}$ such that $\varphi_{y}^{-1} A = A$ up to a null-set for all $y \in \mathbb{R}^{d}$. \\\\
 
   \begin{assumption} \label{stationaryergodicassumption} 
For every fixed $t_{0},x_{0},u$ the process $L(t_{0},x_{0},\cdot, u, \omega)$ is stationary ergodic  in $\Omega$. That is to say there is an ergodic transformation group $ \{ \varphi_{x} \colon \Omega\to\Omega\}_{x\in \mathbb{R}^d}$ such that, for every $x,y$ and $t_{0},x_{0},u$ fixed

 \begin{equation}
 L(t_{0},x_{0},x+y, u, \omega)=L(t_{0},x_{0},x, u,  \varphi_{y} \omega).
 \end{equation}
 \end{assumption}

For the remaining assumptions, we make the following definitions 
   \setlength{\jot}{18pt}
  \begin{align}
  U^{R}:&= \lbrace u \: : \: |u| \leq R \rbrace, \\
  \mathcal{U}^{R}& = \{ u \in \mathcal{U},  \: \: \: \|u\|_{\infty } \leq R \}.
  \end{align}
and use this to define the estimates on the cost function and state dynamics

  \begin{align}
  L^{*}(R)&: = \sup_{t,x,y, u \in U^{R}, \: \omega \in \Omega} |L \left( t,x,y, u,\omega \right) |,\\
  L_{\inf} (R)&: = \inf_{t,x,y, |u| \geq R, \:\omega \in \Omega } L \left( t,x,y, u,\omega \right),\\ 
  f^{*}(R)&: = \sup_{x, u \in U^{R} } |f(x,u)|.
  \end{align}

\begin{assumption} \label{finiteproblem}
For every $R>0$, the terms $L^{*}(R)$, $L_{\inf}(R)$ and $f^{*}(R)$ are finite.
\end{assumption}

 \begin{assumption}\label{Lip-f}
 $f$ is uniformly Lipschitz continuous on each $\mathbb{R}^{d} \times U^{R}$ with Lipschitz constant $\|f\|^{R}_{Lip}$. 
 \end{assumption}
 
  \begin{assumption}\label{Lip-L}
 
 For every $R>0$, there exists a constant $\|L\|_{Lip}^{R}$ and a continuous function $m_{L}^{R} : [0,\infty) \rightarrow [0,\infty)$ with $m_{L}^{R}(0)=0$ such that for every $x_{i}, y_{i} \in \mathbb{R}^{d}$ and $u_{i} \in U^{R}$, $t_{i} \in [0,T]$,  for $i \in \{1,2 \},$
 \begin{equation}
\big | L(t,x_{1},y_{1},u_{1},\omega) -  L(t,x_{2},y_{2},u_{2},\omega) \big |  \leq m_{L}^{R}(|x_{1}-x_{2}|+|y_{1}-y_{2}|) + \|L\|_{Lip}^{R} (|u_{1}-u_{2}|)
 \end{equation}
 
\noindent almost surely in $\Omega$. Furthermore we assume that there exists a $\|L\|_{Lip}$ such that
 \begin{equation}
|L \left( t_{1},x, y, u,\omega) - L(t_{2},x,y, u,\omega \right) | \leq \|L\|_{Lip} (|t_{1}-t_{2}|) 
 \end{equation}

\noindent for all $t,s \in [0,T] $ and $x,y,u \in \mathbb{R}^{d}$ a.s. in $\Omega$. 
  \end{assumption}
  
  \begin{assumption} \label{Terminalcostcont}
 $\psi$ is bounded and there exists a $m^{\psi} : [0,\infty) \rightarrow [0,\infty)$ with $m_{\psi}(0)=0$ such that
 
 \begin{equation}
\big | \psi(x) - \psi(y) \big | \leq m^{\psi}( | x-y|) 
 \end{equation}
 
  \end{assumption}
  
 \begin{assumption}  \label{zeroisinrange}
  For some $\widetilde{M} > 0$, for every $x \in \mathbb{R}^{d}$ there exists a  $u \in U^{\widetilde{M}}$ such that
\begin{equation}
f(x,u)=0. 
\end{equation}

\end{assumption}

 \begin{assumption} \label{Growthassumption}
 There exists an $L_{*}(u) \leq L_{\inf}( | u| )$  and $\lambda > 0$ such that
 \begin{equation} \label{part1GA}
 \frac{L_{*}(u )}{f^{*}(|u|+ \lambda)}=:\gamma(u) \rightarrow \infty \:\:\: \text{as} \:\: |u| \rightarrow \infty.
 \end{equation}

\noindent Furthermore there exists a function $\Theta : \mathbb{R}^{d}  \rightarrow \mathbb{R}^{d}$ such that, for every $x,u_{1},u_{2} \in \mathbb{R}^{d}$

\begin{equation} \label{part2GA}
L_{*}(u_{1})-L_{*}(u_{2}) \geq \Theta(u_{2}) \cdot \big [ f(x,u_{1}) - f(x,u_{2}) \big ] 
\end{equation}

 \end{assumption}

For $\eta >0$ we denote the open ball of radius $\eta$ around $x$ as $B_{\eta}(x)$.

\begin{assumption} \label{implicitfunctionassumption}

For any given $R > 0$ and every $ x_{0} \in \mathbb{R}^{d}$ and $u \in U^{R}$, there exists a $\eta(R) >0$ independent of $x_{0}$ and a Lipschitz continuous function

\begin{equation}
H:  B_{\eta(R)}(x_{0}) \times B_{\eta(R) } \big ( f(x_{0},u) \big ) : \rightarrow  B_{\eta(R) } (u)
\end{equation}

\noindent such that $H(x_{0},u)=u$ and for any $ (x,v) \in B_{\eta(R)}(x_{0}) \times B_{\eta(R) } \big ( f(x_{0},u) \big ) $ 

\begin{equation}
f( x, H(x,v)) = v.
\end{equation}
 This function $H$ depends on $x_{0}$ and $u$, which we omit in our notation. We assume this family of functions $H$ parametrised by elements of $ \mathbb{R}^{d} \times U^{R}$ is  Lipschitz equicontinuous, and denote this modulus of continunity by $\|H\|^{R}_{Lip}$.
\end{assumption}

\begin{assumption} \label{convexassumption}
For each $R>0$  and every $x \in \mathbb{R}^{d}$, $f(x,U^{R})$ is a convex set.
\end{assumption}

 Our main results concerning homogenization of control problems and  Hamilton Jacobi equations  are as follows.
 
\begin{theorem} \label{maintheorem1}

Under Assumptions 1-\ref{convexassumption} there exists an effective Lagrangian $\widetilde{L}$ such that
\begin{equation}
\lim_{\epsilon \rightarrow 0} V_{\epsilon}(t,x) = \widetilde{V} (t,x)
\end{equation}

\noindent uniformly on compact sets, where $\widetilde{V}$ is the value function corresponding to the cost functional 

 \begin{equation}
 \widetilde{J} (t,x,u)=\int_{t}^{T} \widetilde{L}  \left ( s,x(s), u(s) \right ) ds + \psi(x(T)) 
 \end{equation}
 
\noindent and state dynamics

 \begin{equation} 
\frac{dx(s)}{ds} =f(x(s),u(s)) .
\end{equation}

\end{theorem}

\begin{theorem} \label{maintheorem2}

Under Assumptions 1-\ref{convexassumption} there exists an effective  Hamiltonian $\widetilde{H}$ such that
\begin{equation}
\lim_{\epsilon \rightarrow 0} V_{\epsilon}(t,x) = \widetilde{V} (t,x)
\end{equation}

\noindent uniformly on compact subsets, where $\widetilde{V} (t,x)$ is the unique viscosity solution to the Hamilton Jacobi equation

\begin{equation}
- \frac{\partial \widetilde{V}(t,x)}{\partial t} + \widetilde{H}(t,x,D_{x} \widetilde{V} )=0, \:\:\:\: V(T,x)=\psi(x).
\end{equation}

\end{theorem} 

Assumption $1$ is a standard stationary ergodic assumption. Assumption $2$ is a similarly standard assumption to ensure the solution to the integral equation

\begin{equation}
 x(t)-x(t_{0})=\int_{t_{0}}^{t} f(x(s),u(s)) ds
 \end{equation}
 
 \noindent is well defined for each $u \in \mathcal{U}$.\\
 
The continuity conditions in  Assumption $3$ are  used to prove effective boundedness as per Lemmas $2.2$ and $2.3$ below. In fact this is the only place that Lipschitz continunity is needed and the proof of Lemmas \ref{effectiveboundednessofvaluelemma} and \ref{Boundedness lemma2} extensively rely on it.  \\

The combination of Assumption \ref{zeroisinrange} and \ref{implicitfunctionassumption} imply that there exists a $\delta >0$ such that at each $x$, for every $v \in \mathbb{S}^{d}$ (the d-dimensional sphere) there is a $u \in U^{M}$ such that

\begin{equation} \label{hereisdelta}
f(x,u)=\delta v
\end{equation}

\noindent for $M:= \widetilde{M} + \eta(\widetilde{M}).$ This is one form of the `flexibility' mentioned in section 1.2  which is important to stipulate. The terms  $M$ and $\delta$ will be referred to throughout. \\

 Assumption \ref{Growthassumption} is analogous to various coercive assumptions in the Calculus of variations setting. The presence of the $\lambda > 0$ makes this a mildly stronger assumption, however is satisfied in ``many cases''.  For example, in the calculus of variations with $L_{*}(u) \geq c|u|^{\beta} -C$   for $\beta>1$ and appropiate constants $c$ and $C$.  The role of the second statement in this assumption is to prove that the homogenized Lagrangian satisfies the same coercive assumption as the original Lagrangian. It is essentially a generalization of convexity. In fact in the calculus of variations if we assume that $L_{*}(u)$ is a differentiable convex function, then the second statement of Assumption \ref{Growthassumption} is satisfied by setting $\Theta ( u): = \nabla L_{*} (u) $. This is discussed further in section 2.5. \\
 
% \begin{equation}
% c_{1} |u|^{\alpha} -C_{1} \leq |L(t,x,y,u)|  \leq c_{2} |u|^{\alpha} -C_{2}
% \end{equation}
% 
 The implication of Assumption $5$ is that we can essentially compactify the control set. This is a key point in the analysis. Assumption $6$ gives us flexibility as in \eqref{hereisdelta} but it in fact implies more.
 For example, given any control, we can always find another control which follows the same trajectory but at a higher speed. Its formulation is motivated by an implicit function theorems. A natural setting where Assumption \ref{implicitfunctionassumption} holds is when $f$ is continuously differentiable in the control variable with an Jacobian whose inverse is uniformly bounded on  $\mathbb{R}^{d} \times U^{R}$ for each $R>0$. The function $H$ appearing in Assumption \ref{implicitfunctionassumption} will also be referenced throughout, with the dependence on $x_{0},u$ omitted - they will always be clear from the context. Similarly to Assumption $3$, the Lipschitz continuity of $H$ is necessary for the proof of Lemmas \ref{effectiveboundednessofvaluelemma} and \ref{Boundedness lemma2}. \\
 
 Assumption \ref{implicitfunctionassumption} and \ref{convexassumption} will imply the existence of the $\widetilde{u}$ as in equation \eqref{tildeu}, which is essential to preserve the state dynamics. Although restrictive, Assumption \ref{convexassumption}, is a common assumption in the theory of optimal control. A typical result is that under Assumption \ref{convexassumption}, suitable regularity conditions, and compactness (or effective compactness as per Lemma \ref{effectiveboundednessofvaluelemma}) of the control set, an optimal control exists.   We refer the reader to \cite{a.w.j.stoddart1967}, \cite{a.f.filippov1962} and \cite{roxin1962} for more details, though we emphasise that the existence of an optimal control does not play a role in the homogenization process. \\

\begin{remark}
Let us compare our assumptions to those in~\cite{souganidis1999stochastic,Rezakhanlou2000} in their setting. Thus we focus on the case of calculus of variations, which makes Assumptions~\ref{Lip-f},~\ref{zeroisinrange} and~\ref{implicitfunctionassumption} trivial. Assume that $L$ is independent of the first and second arguments. 
Then Assumption~\ref{Growthassumption} is equivalent to a mean coercivity condition which states that $L_{*}(u)$ can be taken as a convex function satisfying $\lim_{|u| \to \infty} L_*(u)/|u|=\infty$, which both~\cite{souganidis1999stochastic} and~\cite{Rezakhanlou2000} assume. But in fact additional polynomial growth conditions are assumed in~\cite{souganidis1999stochastic} (upper and lower bounds) and~\cite{Rezakhanlou2000} (upper bound), which we do not require. 
Concerning Assumptions~\ref{Lip-L}, note first that in this case, we may assume that $L$ is convex in $u$ since its convex biconjugate leads to the same Hamiltonian and hence same value function. Combined with the above polynomial upper bounds, the local Lipschitz continuity in $u$ follows. Hence it is also implicitly assumed in~\cite{souganidis1999stochastic} and~\cite{Rezakhanlou2000}. Therefore our regularity assumption is no stronger than theirs, and~\cite{Rezakhanlou2000} moreover assumes continuous differentiability of $L$ in $u$. 
\end{remark}

The assumptions we impose on the state dynamics are strong however is more general then the calculus of variations. For example,  we may wish to consider when the maximum speed is bounded by a constant $C$. This is fundamentally different to the calculus of variations which places no such restrictions. To this end consider the following state function

\begin{equation}
f(x,u) = \frac{ C u} {\sqrt{|u|^{2}+1} }
\end{equation}

Observe that

\begin{equation}
|f(x,u)|    \leq C.
\end{equation}

For $v=f(x,u)$ the inverse of this is

\begin{equation} \label{inverse}  
u =  \frac{ v} { \sqrt{ 1 - \frac{ |v|^{2}}{C^{2}}} } 
\end{equation}

 The state dynamics satisfy Assumptions $3,6,8$ and thus may be considered to fall under our framework. Suppose that $L_{*}(u) = |u|^{\beta} -C$ for some $C > 0 $. Then the first statement in Assumption \ref{Growthassumption}, as given in equation \eqref{part1GA}, is true whenever $\beta >0$.  However we shall find in section $2.5$ that the  second statement is satisfied only if $\beta \geq 1$. \\

 \section{Proof of Main results}
 
 \subsection{Approximation scheme}
 In this subsection we define the approximation scheme which is the key construction in the proof of main result. The discretization is based on defining a discrete cost function by locally optimizing over a given terminal condition.  For $\tau>0$ let $t_{0}=t$ and $t_{i}=t_{i-1}+\tau \wedge T$ until the first $N$ such that $t_{N}=T$. This sequence $\{ t_{i} \}_{i=1}^{N}$ defines a partition of $[t,T]$. We recall the definition
 
  \begin{equation} 
L_{\tau, \epsilon}(t_{0},x_{0},\widetilde{u},\omega): = \inf_{u(s) \in \Pi_{t_{0},x_{0},\tau,\widetilde{u}}}  \int_{t_{0}}^{t_{0}+\tau} L\left(t_{0},x_{0},\frac{x(s)}{\epsilon}, u(s),\omega \right) ds
\end{equation}

\noindent where

\begin{equation}
\Pi_{t_{0},x_{0},\tau,\widetilde{u}}=\left\{ u \in \mathcal{U} \: : \:\int_{t_{0}}^{t_{0}+\tau} f(x_{0},u(s))ds= \tau f(x_{0},\widetilde{u}) \right\}. 
\end{equation} 

 Note that in these objects, all the macroscopic variables are frozen. This is to ensure stationarity  in preparation  of utilizing the sub-additive ergodic theorem.\\ 
 
  Still keeping $t_{0},x_{0}$ fixed, consider any sequence finite sequence control vectors $\lbrace \widetilde{u}_{i}  \rbrace_{i=1}^{N} \in \mathbb{R}^{d}$ and the uniquely defined $x_{i}$.

\begin{equation}
x_{i+1}=x_{i} + \tau f(x,\widetilde{u}_{i}) 
 \end{equation}

   Let $\mathcal{U}_{\tau}$ denote the set of all such sequences $\lbrace \widetilde{u}_{i}  \rbrace_{i=1}^{N} \in \mathbb{R}^{d}$.  We may now use $\mathcal{U}_{\tau}$ as an approximation to our control problem. 
For $\widetilde{u}  \in \mathcal{U}_{\tau}$ define an  approximate effective cost functional

 \begin{equation} \label{Approximatecostfunction}
 J_{\tau,\epsilon}(t,x,\widetilde{u}, \omega) = \sum_{i=0}^{N-1} L_{\tau, \epsilon}(t_{i},x_{i}, \widetilde{u}_{i}, \omega) + \psi(x_{N}) 
 \end{equation}

\noindent and approximate effective value function

  \begin{equation} \label{Approximatevaluefunctionsum}
 V_{\tau, \epsilon}(t,x, \omega) = \inf_{\widetilde{u} \in \mathcal{U}_{\tau}}  J_{\tau, \epsilon}(t,x,\widetilde{u}, \omega).
 \end{equation}
 
We also define the non-stationary version of $L_{\tau,\epsilon}$ as
 
 \begin{equation} \label{nonstationaryone}
\widehat{L}_{\tau, \epsilon}(t_{0},x_{0},\tau,\widetilde{u},\omega) := \inf_{u(s) \in \widehat{\Pi}_{t_{0},x_{0},\tau,\widetilde{u}}}  \int_{t_{0}}^{t_{0}+\tau} L \left( s,x(s),\frac{x(s)}{\epsilon}, u(s),\omega \right) ds,
\end{equation}

\noindent where

 \begin{equation}
\widehat{\Pi}_{t_{0},x_{0},\tau,\widetilde{u}} = \left\{ u \in \mathcal{U} \: : \:\int_{t_{0}}^{t_{0}+\tau} f(x(s),u(s))ds= \tau f(x_{0},\widetilde{u}) \right\}. 
\end{equation} 
A crucial point is that if $u \in \Pi_{t_{0},x_{0},\tau,\widetilde{u}}$, then, provided that $\tau$ and $f(x_{0},\widetilde{u})$  are sufficiently small, namely $|\tau f(x_{0},\widetilde{u})| \leq \eta(|\widetilde{u}|)$, then by Assumption \ref{implicitfunctionassumption},  there will exist a $\widehat{u} \in  \widehat{\Pi}_{t_{0},x_{0},\tau,\widetilde{u}}$ which traces out the same trajectory as u . Hence if we are in this regime, $\widehat{L}_{\epsilon}$ is always well defined.

 \subsection{Technical Lemmas}

	Throughout the proof of the main theorem, we will require several technical lemmas. The proof of these is relegated to Section~\ref{Tech}. The first of these is regarding the use of approximating controls with piecewise constant dynamics.\\
 
  We consider a general control problem $\bar{f}(x,u)$ and Lagrangian $\bar{L}(t,x,u)$. For this control problem we define $\mathcal{S}$ be the set of controls such that $\bar{f}(x(s),u(s))$ is a step function. That is, $\mathcal{S}$ is the set of controls such that for some intervals $\{ [t_{i},t_{i+1} ) \}_{i=1}^{N} $ partitioning $[t_{0},\bar{T}]$ and values $\{ v_{i} \}_{i=1}^{N} $,

\begin{equation}
 \bar{f}(x(r),u(r))=\sum_{i=0}^N v_{i} 1_{[t_{i}, t_{i+1})}(r).
\label{step-f}
\end{equation}
 
  However, we always assume that $\bar{f}$ satisfies Assumption \ref{implicitfunctionassumption} and the bounds $f^{*}$, $L^{*}$ as given in the assumption. The following lemma will be applied to multiple control problems.
 
 \begin{lemma} \label{StepfunctionLemma}
Let $\bar{J}$ and $\bar{V}$ be the cost and value function of a control problem on time $[0,\bar{T}]$ with Lagrangian $\bar{L}$ and state dynamics $\bar{f}$. Assume that $\bar{L}$ and $\bar{f}$ satisfy Assumptions $2,3$ and $6$.  % on every $[\bar{t}_{k}, \bar{t}_{k+1}] \times \mathbb{R}^{d} \times U^{R}$ for each $R$.\\
 Then for every $t \in [0,\bar{T}]$ and $x \in \mathbb{R}^{d}$ 
\begin{equation}
\bar{V}(t,x):=\inf_{u \in \mathcal{U}} \bar{J}(t,x,u) = \inf_{u \in \mathcal{S}} \bar{J}(t,x,u).
\end{equation}

\noindent The same is true for the control problem with terminal constraint $x(\bar{T})=x$.
\end{lemma}
 
 The second lemma argues that we can, without loss of generality, restrict the control space to a compact set. To formalize this we consider the cost functions obtained by restricting the admissible controls to $\mathcal{U}^{R}$ for a fixed $R>0$ and hence define
  
   \begin{equation}
    V^{R}_{ \epsilon}(t,x,\omega):= \inf_{u \in \mathcal{U^{R}}} J_{\epsilon} (t,x,u, \omega)
   \end{equation}
   
\noindent and the objects

\begin{align} 
L^{R}_{\tau, \epsilon} (t_{0},x_{0},\widetilde{u},\omega): &= \inf_{u(s) \in \Pi_{t_{0},x_{0},\tau,\widetilde{u}} \cap \mathcal{U}^{R}}  \int_{t_{0}}^{t_{0}+\tau} L \left( t_{0},x_{0},\frac{x(s)}{\epsilon}, u(s),\omega\right) ds
\end{align}
and
\begin{align} 
\widehat{L}^{R}_{\tau, \epsilon} (t_{0},x_{0},\widetilde{u},\omega): &= \inf_{u(s) \in \widehat{\Pi}_{t_{0},x_{0},\tau,\widetilde{u}} \cap \mathcal{U}^{R}}  \int_{t_{0}}^{t_{0}+\tau} L \left( t_{0},x_{0},\frac{x(s)}{\epsilon}, u(s),\omega\right) ds
\end{align}

Similarly, for the discrete control problem we define 

\begin{equation}
\mathcal{U}^{R}_{\tau} = \{ \widetilde{u} \in \mathcal{U}_{\tau} \: : \:  \forall i \leq N, \widetilde{u}_{i}  \leq R\}
\end{equation}

\noindent and the value function 

\begin{equation}
V^{R}_{\tau}(t,x,\omega) = \inf_{\widetilde{u} \in \mathcal{U}^{R}_{\tau}} \widetilde{J}_{\tau, \epsilon}(t,x,u,\omega)
\end{equation}
 
\begin{lemma} \label{effectiveboundednessofvaluelemma}
There exists an $K >0$ such that for all $t \in [0,T]$ and $x \in \mathbb{R}^{d}$, $\epsilon>0$ and almost every $\omega$,

\begin{equation} \label{effectiveboundednessofvalue}
V_{\epsilon}^{K}(t,x,\omega)=V_{\epsilon}(t,x,\omega)
\end{equation}

and 

\begin{equation} \label{effectiveboundednessofvaluediscrete}
V_{\tau, \epsilon}^{K}(t,x,\omega)=V_{\tau,\epsilon}(t,x,\omega).
\end{equation} 

\end{lemma}

\begin{lemma} \label{Boundedness lemma2}

For $\tau >0$ and $\widetilde{u} \in \mathbb{R}^{d}$ such that $\tau  f(x_{0},\widetilde{u}) \leq \eta(|\widetilde{u}|)$, there  exists a $K=K(\widetilde{u})$ such that for all $t_{0} \in [0,T]$ and $x_{0} \in \mathbb{R}^{d}$, $\epsilon>0$ and almost every $\omega$,

\begin{align} 
L_{\tau, \epsilon}(t_{0},x_{0},\widetilde{u},\omega)&=L^{K}_{\tau,\epsilon}(t_{0},x_{0},\widetilde{u},\omega) \label{effectiveboundednessofLagrange} 
\end{align}
and
\begin{align} \widehat{L}_{\tau, \epsilon}(t_{0},x_{0},\widetilde{u},\omega)&=\widehat{L}^{K}_{\tau,\epsilon}(t_{0},x_{0},\widetilde{u},\omega).
\end{align}
\end{lemma}

\noindent These $K$ will be referred to throughout the paper.

\begin{remark} These lemmas are not to be taken for granted. In the context of more general control theory, neither  Lemma \ref{StepfunctionLemma}, nor Lemmas \ref{Boundedness lemma2}, \ref{effectiveboundednessofvaluelemma} or even the fact that we are from the outset minimizing over controls in $L^{\infty} \left(  [0,T] \rightarrow \mathbb{R}^{d} \right)$, is something which is justified without strict assumptions. In the calculus of variations setting, we might more generally minimize over the set of measurable controls, $\mathcal{M}$. However there are examples in which 

\begin{equation}
\inf_{u \in \mathcal{M}}  J(t,x,u) <  \inf_{u \in \mathcal{U}}  J(t,x,u).
\end{equation}
\indent Such examples are said to exhibit the Lavrentiev  phenomenon. This shows that a careful choice of  the space of minimizers is required.  A particular one-dimensional example is given in \cite{JohnBall1985}, which shows that for the Lagrangian 

\begin{equation} \label{Lavrentievphenomenon}
L(t,x,u)= (t^{4}-x^{6})^{2} |u|^{s} + \epsilon |u|^{2}
\end{equation}
on the time interval $[-1,1]$, the  Lavrentiev  phenomenon is exhibited for some $\epsilon >0$ and $s \geq 27$, under the boundary conditions $x(-1)=k_{1}, x(1)=k_{2}$ for some $k_{1},k_{2}$.
Furthermore the Lagrangian in \eqref{Lavrentievphenomenon} satisfies every assumption in our hypothesis of Lemma \ref{Boundedness lemma2} except the uniform Lipschitz continuity with respect to time, as given by the Lipschtiz constant $\|L\|_{Lip}$. If we attempted to replicate the proof of Lemma \ref{Boundedness lemma2}, as given in Section $3$ for the Lagrangian in \eqref{Lavrentievphenomenon}, it would fail for that very reason. Thus we see that the results of this paper rely delicately on the assumptions.
\end{remark}

\subsection{Continuity of the discretized control problem and value function}
\indent In this subsection we study the error of using a stationary approximation to the control problem and prove continuity properties of $L_{\tau,\epsilon}$ and $V_{\epsilon}$. These lemmas will then be applied in the next subsection to analyse the continuity of the effective Lagrangian $\widetilde{L}$ and to show that it is almost surely constant in $\Omega$.\\

Lemma \ref{Boundedness lemma2} is key for the following  two lemmas. The first of these will be used, albeit in a different form, to show convergence of the Value functions in Subsection $2.5$. It illustrates the error obtained if we use a stationary approximation.

\begin{lemma} \label{whythisworks}
For $\tau | f^{*}(K)| \leq |\eta(K)| $ where $K=K(\widetilde{u})$ as in Lemma \ref{Boundedness lemma2}, we have
\begin{multline}
  \big | L_{\tau, \epsilon} (t_{0},x_{0},\widetilde{u},\omega) - \widehat{L}_{\tau, \epsilon} (t_{0},x_{0},\widetilde{u},\omega) \big | \\  \leq    \tau^{2} \big [ \|L\|_{Lip} + f^{*} (K)  \|L\|_{Lip}^{K+\eta(K)} \|H\|_{Lip}^{K+\eta(K)}  \big ] + \tau m_{L}^{ K+\eta(K)}(\tau f^{*}(K)). 
\end{multline}

 \end{lemma}
 \begin{proof}
Take any control $u \in \widehat{\Pi}_{t_{0},x_{0},\tau,y}$. Without loss of generality, we can assume that $\|u\|_{\infty} \leq K$. This implies that $\sup_{s \in [t_{0},t_{0}+\tau]} |x(s)-x_{0}| \leq \tau f^{*}(K)$. Therefore we may use Assumption \ref{implicitfunctionassumption} and deduce that for each $s \in [t_{0},t_{0}+\tau]$ there will exist a $\bar{u} \in  \mathcal{U}^{K+\eta(K)}$ such that
 
 \begin{equation}
 f(x_{0},\bar{u}(s)) =f(x(s),u(s))   
 \end{equation}
 
 \noindent and
 
 \begin{equation}
 |\bar{u}(s) - u(s)| \leq \|H\|^{K+\eta(K)}_{Lip} (|\tau f ^{*}(K)| )
 \end{equation}
this control  follows the same trajectory, in that $\bar{x}(s)=x(s)$. Then 
 \begin{multline} 
  \left | \int_{t_{0}}^{t_{0}+\tau} L \left( t_{0},x_{0}, \frac{x(s)}{\epsilon}, \bar{u}(s), \omega \right) ds - \int_{t_{0}}^{t_{0}+\tau} L \left( s,x(s), \frac{x(s)}{\epsilon}, u(s), \omega \right) ds \right| \\ 
  \leq  \tau^{2} \big [ \|L\|_{Lip} + f^{*} (K)  \|L\|_{Lip}^{K+\eta(K)} \|H\|_{Lip}^{K+\eta(K)}  \big ] + \tau m_{L}^{ K+\eta(K)}(\tau f^{*}(K)) 
  \end{multline}
  
 as the control was arbitrary, this proves that 
  \begin{multline}
\widehat{L}_{\tau, \epsilon} (t_{0}.x_{0},\widetilde{u}, \omega) +\tau^{2} \big [ \|L\|_{Lip} + f^{*} (K)  \|L\|_{Lip}^{K+\eta(K)} \|H\|_{Lip}^{K+\eta(K)}  \big ] + \tau m_{L}^{ K+\eta(K)}(\tau f^{*}(K))  \\ \geq L_{\tau, \epsilon} (t_{0}.x_{0},\widetilde{u},\omega).
  \end{multline} 
  The converse is proved in the same fashion and this  completes the proof.
 \end{proof}

We turn our attention to obtain a modulus of continuity on $L_{\tau, \epsilon}$. 

\begin{lemma} \label{contlemma}

For a given $t_{0} \in [0,T]$, $x_{0},y \in \mathbb{R}^{d}$, $\epsilon>0$ and $\tau_{1}, \tau_{2} > 0$ and $\widetilde{u}_{1}, \widetilde{u}_{2} \in \mathbb{R}^{d}$  are such that $\tau  > \delta^{-1} \big ( |\tau_{1} f( x_{0}, \widetilde{u}_{1}) - \tau_{2} f(x_{0}, \widetilde{u}_{2})| \big ) +\epsilon y $. Then define 

\begin{equation} \label{defofbeta}
   \beta_{1}=\frac{\tau_{1}}{\tau_{2} -\delta^{-1} \big ( |\tau_{1} f( x_{0}, \widetilde{u}_{1}) - \tau_{2} f(x_{0}, \widetilde{u}_{2})| +\epsilon y \big )}, 
   \end{equation}
   \begin{equation}
   \beta_{2}=\frac{\tau_{2}}{\tau_{1} -\delta^{-1} \big ( |\tau_{2} f( x_{0}, \widetilde{u}_{2}) - \tau_{1} f(x_{0}, \widetilde{u}_{1})| +\epsilon y \big )}, 
   \end{equation}
   
   and
   
   \begin{equation}
   \widetilde{\beta}^{*} := \textnormal{argmax}_{\beta_{1},\beta_{2}} | \beta_{i} -1| 
   \end{equation}
   
    \begin{equation}
   \widetilde{\beta}_{*} :=  \textnormal{argmax}_{\beta_{1},\beta_{2}} | 1/ \beta_{i} -1|, 
   \end{equation}
   
 \noindent $\tau:= \tau_{1} \wedge \tau_{2}$ and $K:=K(\widetilde{u}_{1}) \wedge K(\widetilde{u}_{2})$, where $K(\widetilde{u}_{1})$ and $K(\widetilde{u}_{2})$ are as in Lemma \ref{Boundedness lemma2}. Assume $ |\widetilde{\beta}^{*} -1| \leq \eta(K)$.  Then for every $\omega \in \Omega$ almost surely,
 \begin{multline}
 \big |L_{\tau_{1}, \epsilon} (t_{0},x_{0}, u_{1},\omega) -L_{\tau_{2}, \epsilon} (t_{0},x_{0},u_{2},\omega) \circ \varphi_{y}  |
 \\
  \leq \tau L^{*}(K) |1 / \beta_{*} -1 |  + \tau \|L\|_{Lip}^{K+\eta(K)} \|H\|^{K +\eta(K)}_{Lip} f^{*}(K)| \beta^{*} - 1 | +\frac{L^{*}(M)}{\delta}   ( \|f\|^{K}_{Lip} |u_{1} - u_{2}|+\epsilon y).
 \end{multline}

\end{lemma}

\begin{proof}

Without any loss of generality we take $t_{0}=0$. Fix any control $u \in \Pi_{0,x_{0},\tau_{1},\widetilde{u}_{1}}$, and assume $\|u\|_{\infty} \leq K$.  We will adopt a similar strategy of defining a control, $\bar{u}$ which traces out the same trajectory. Define $\bar{u}(s) \in \Pi_{0,x_{0},\tau_{2},\widetilde{u}_{2}}$ as follows. For $s \in [0, \tau_{1}- \epsilon \delta^{-1} |y|]$ define $\bar{u}$ such that

\begin{equation}
f(x_{0},\bar{u}(s)) = \delta \frac{ y}{|y|},
\end{equation}
 this $\delta$ is as in equation \eqref{hereisdelta}. The trajectory of $\bar{u}$, denoted $\bar{x}$,  will satisfy $\bar{x}(\epsilon \delta^{-1}  |y|) = x_{0}$.
%This is due to the fact that we may rescale $\mathcal{L}_{\epsilon}$ and write
% 
%  \begin{equation} \label{scaling}
% \mathcal{L}_{\epsilon}(t_{0},x_{0},h,u,\omega) \varphi_{y} = \inf_{u \in \Pi_{t,x,h/ \epsilon,u/ \epsilon }} \epsilon  \int_{t/\epsilon}^{(t+h)/\epsilon} L(t_{0},x_{0},\widetilde{x}(\widetilde{s}) + \epsilon y, u(\widetilde{s}),\omega)  d \widetilde{s}
% \end{equation}  
% 
% and so the trajectory of $\bar{u}$, denoted $\bar{x}$  will satisfy $\bar{x}(\epsilon \delta^{-1} y / |y|) = x_{0}$
% 
For  $s \in [\epsilon \delta^{-1} |y|, \tau_{2} - \delta^{-1}| \tau_{1} f(x_{0},u_{1}) - \tau_{2} f(x_{0},u_{2})| ] $ define $\bar{u}(s)$ as the solution to 

\begin{equation} 
f(x_{0},\bar{u}(s))=\beta_{1} f (x_{0},u(\beta s)).
\end{equation}
 This will reach $\tau_{1} f(x_{0}, u_{1})$ by time $ \tau_{2} - \delta^{-1} |\tau_{1} f(x_{0},u_{1}) - \tau_{2} f(x_{0}, u_{2})|$, and the existence of such a control is guaranteed by Assumption \ref{implicitfunctionassumption} and our assumption $|\beta^{*}-1| \leq  \eta(K)$. in the statement of the lemma. During this time note that 
 
 \begin{equation}
 | \bar{u}(s)- u (\beta_{1} s) | \leq \|H\|_{Lip}^{K+\eta(K)} ( f^{*}(K)|\beta_{1} -1 |).
 \end{equation}
 
 and $\bar{x(t)} = x(\beta_{1} t)$. \\
 For $s \in [  \tau_{2} - \delta^{-1} |f(x_{0},\widetilde{u}_{1})-f(x_{0},\widetilde{u}_{2})|,  \tau_{2}  ] $ we define $\bar{u}(s)$ such that
 
 \begin{equation}
 f(x_{0}, \bar{u}(s)) = \delta \frac{  f(x_{0},\widetilde{u}_{1})-f(x_{0},\widetilde{u}_{2}) }{| f(x_{0},\widetilde{u}_{1})-f(x_{0},\widetilde{u}_{2}) |}
 \end{equation}
which ensures $\bar{x}(\tau_{2})=x_{0}+\tau_{2} f(x_{0},\widetilde{u}_{2})$. Then,

   \begin{equation}
\begin{split}
   &\left|\int_{0}^{\tau_{1}} L\left( 0,x_{0},\frac{x(s)}{\epsilon}, u(s), \omega \right ) ds  - \int_{0}^{\tau_{2}} L \left( 0,x_{0},\frac{\bar{x}(s)}{\epsilon}, \bar{u}(s), \omega \right) ds \right| \\ 
&\quad   \leq \left|  \int_{0}^{\tau_{1} } L \left( 0,x_{0},\frac{x(s)}{\epsilon}, u(s), \omega \right)ds   - \frac{1}{\beta_{1}}  \int_{0}^{ \tau_{1} } L \left( 0,x_{0},\frac{x(s)}{\epsilon}, \bar{u} ( \beta_{1}^{-1} s), \omega \right)ds  \right|  \\
&\qquad +\int_{0}^{\epsilon \delta^{-1} y / |y|} L \left( 0,x_{0}, \frac{\bar{x}(s)}{\epsilon} , \bar{u}(s), \omega \right) ds +  \int^{\tau_{2}}_{\tau_{2} - \delta^{-1} |\tau_{1} f(x_{1},u_{1}) - \tau_{2} f(x_{2}, u_{2})|} L \left( 0,x_{0},\frac{x(s)}{\epsilon},\bar{u}(s), \omega  \right) ds.  \\
&\quad     \leq 
     \tau L^{*}(K) |1 / \beta_{1} -1 |  + \tau \|L\|_{Lip}^{K+\eta(K)} \|H\|^{K+\eta(K)}_{Lip} f^{*}(K)|(\beta_{1} - 1) | + \frac{L^{*}(M)}{\delta}   ( \|f\|^{K}_{Lip} |u_{1} - u_{2}|+\epsilon y)
\end{split}   
\end{equation}
The control $u \in \Pi_{0,x_{0},\tau_{1},\widetilde{u}_{1}}$ was arbitrary.  We can repeat the proof for an arbitrary  $\hat{u} \in \widehat{\Pi}_{0,x_{0},\tau_{2},\widetilde{u}_{2}}$ and acquire a similar bound for involving $\beta_{2}$, and then taking the maximum, arrive at the conclusion.
    \end{proof}

 Finally we also wish to get a uniform continuity bound on $V_{\epsilon}(t,x,\omega)$ uniform in $\epsilon$. For this proof, it similarly essential that we have effective compactness stated in Lemma \ref{effectiveboundednessofvaluelemma}. 

 \begin{lemma} \label{Value function cont}
 Let $K$ be as in as in Lemma \ref{effectiveboundednessofvaluelemma}.  Then for all  $\epsilon > 0$, almost every $\omega$ and $x_{i} \in \mathbb{R}^{d}$,  $t_{i} \in [0,T]$, $ i \in \{ 1,2 \}$, 
 
 \begin{equation}
 \begin{split}
 |V_{\epsilon}(t_{1},x_{1}, \omega) - V_{\epsilon}(t_{2},x_{2},\omega)| 
 &  \leq  \frac{L^{*}(M)}{\delta} |x_{1}-x_{2}| + T \|L\|_{Lip} (|t_{1}-t_{2}|+\delta^{-1} |x_{1}-x_{2}|)  \\ & +
  T( |t_{1}-t_{2}| + \delta^{-1} |x_{1}-x_{2}| ) L^{*}(K) + m^{\psi}( |t_{1}-t_{2}|+ \delta^{-1} |x_{1}-x_{2}| ) 
 \end{split}
 \end{equation} 
 
\end{lemma} 

 \begin{proof}
 The concept and argument of the proof is very similar to that of the previous lemma.  If we start from $x_{1}$, we can move to position $x_{2}$ in time $ \delta^{-1} |x_{1}-x_{2}| $. From here we can trace out the same trajectory. The difference that we pick up is accounted as follows
 \begin{enumerate}
 \item The first term is the cost during the time taken for the control to reach $x$. The cost of this is
 
 \begin{equation}
 \int_{t_{1}}^{t_{1} - \delta^{-1} |x_{1}-x_{2}|} L \left(  s,x(s), \frac{x(s)}{\epsilon}, u(s), \omega  \right) ds \leq \frac{L^{*}(M)}{\delta} |x_{1}-x_{2}|
 \end{equation}
 
 for some control $u(s)$ such that 
 
 \begin{equation} 
 f(x(s),u(s))=  \delta \frac{ x_{2}-x_{1} }{|x_{2}-x_{1}|}.
 \end{equation}

 \item The second term accounts for the difference in the time argument of the Lagrangian throughout the imitating control. This is bounded by $\|L\|_{Lip} (|t_{1}-t_{2}|+\delta^{-1} |x_{1}-x_{2}|) $ 
 
 \item The third term follows from the fact that any control attempting to trace out another control which started at $t$, may fall short of time up to a maximum of $|t_{1}-t_{2}|+ \delta^{-1} |x_{1}-x_{2}|$. However over any period of time, we can assume that the Lagrangian is bounded proportional to $L^{*}(K)$
 
 \item The last term comes from the difference in the terminal cost, which by  Assumption \ref{Terminalcostcont} is less then $m^{\psi}( |t_{1}-t_{2}|+ \delta^{-1} |x_{1}-x_{2}| ) $
 
  \end{enumerate}

  \end{proof}

 \subsection{Subadditive Ergodic theorem and effective Lagrangian}
 In this subsection we state a continuous parameter version of the sub-additive ergodic theorem to deduce \eqref{convergenceofsubaddergotheorem} almost surely and define the effective Lagrangian. \\

 Define the set $Q_{c} = \{ (a,b) \in \mathbb{R}^{2} \: : \:  0 \leq a < b \}$ Let $F=\{  F_{a,b} : (a,b) \in Q_{c} \}$ be a family of integrable real valued functions on our probability space $(\Omega, \mathcal{F}, P)$.  Let $\{ \widetilde{\varphi}_{r}  : r \geq 0 \}$ be a family of measure preserving mappings acting on $\Omega$. The family $F_{a,b}$ is called a continuous sub-additive process if the conditions
  
  \begin{enumerate}
  \item[(C1)]  $F_{a,b}  \circ \widetilde{\varphi}_{r} =   F_{a +r, b + r}$  \:\: ($(a,b) \in Q_{c}, \: l \geq 0$ ) 
  \item[(C2)]  $F_{a,l} \leq F_{a,b} + F_{b,l}$ \: \: ($(a,l),(b,l) \in Q_{c}$)
  \item[(C3)]   $\mathbb{E}(F_{0,b}) \geq \gamma b$ for some $\gamma > -\infty$
   \end{enumerate}
are satisfied and $F_{a,b}(\omega)$ is a measurable map of $Q_{c} \times \Omega \rightarrow \mathbb{R}$ with respect to the product $\sigma$-algebra in $Q_{c} \times \Omega$.  The expection is with respect to the measure $P$.\\\\
 For the statement and proof of the following proposition we refer to \cite{krengel1985ergodic}.
  
\begin{prop} \label{Sub-additive ergodic theorem}
Suppose that $\{ F_{a,b} \}_{(a,b) \in Q_{c}} $ satisfies (C1)-(C3), is continuous in $a,b$ and there exists a $K(a,b)$ such that  $|F_{a,b}(\omega)|  \leq K(a,b)$ a.s. \footnote{In the formulation given in \cite{krengel1985ergodic}, the conditions are weaker. $F_{a,b}$ is assumed to be  separable and has integrable oscillations. We refer the reader to \cite{krengel1985ergodic} for further details.}  Then

\begin{equation}
\lim_{b \rightarrow \infty} b^{-1} F_{0,b}
\end{equation}

\noindent exists a.s. and is $\widetilde{\varphi}_{r}$ invariant for all $r \geq 0$.
 \end{prop}

 To apply the sub-additive ergodic theorem to $L_{\tau,\epsilon}$, we need to reintepret our problem as ``long term averages.''\\

 For any given control, $u(s) \in \Pi_{t,x,\tau,\widetilde{u}}$ we have that by substitution $ \widetilde{s}=s / \epsilon$
 
 \begin{equation} \label{scaleddynamics}
 \frac{x(t+\tau)}{\epsilon} =  \frac{x_{0}}{\epsilon}  +   \int_{t/\epsilon }^{ (t+\tau) / \epsilon } f(x_{0} ,u(\epsilon \widetilde{s})) d\widetilde{s} = : \widetilde{x}(t+\tau) .
 \end{equation}
% 
%
%I.e we might consider  the control problem with the state dynamics given by $ \epsilon f(x,u)$ and cost function $ \epsilon L(t_{0},x_{0}, \cdot, \cdot)$. 
After the same substitution $\widetilde{s}$ our object of interest is equivalent to

  \begin{equation} \label{scaling}
L_{\tau, \epsilon}(t_{0},x_{0},\widetilde{u},\omega) = \inf_{u \in \Pi_{t,x,\tau / \epsilon,\widetilde{u}/ \epsilon }} \epsilon  \int_{t/\epsilon}^{(t+\tau)/\epsilon} L(t_{0},x_{0},\widetilde{x}(\widetilde{s}), u(\widetilde{s}),\omega)  d \widetilde{s}.
 \end{equation}  
 The distinguishment between $(\widetilde{x}(s), \widetilde{s})$ and $(x(s), s)$ is superficial, we shall henceforth write $(x(s), s)$.  For our purpose, we then define $F_{a,b}$ as 
 
 \begin{equation}
 F_{a,b} :=  \inf_{u \in \Pi_{t_{0},x_{0}+a f (x_{0}, \widetilde{u}), b-a, \widetilde{u}}} \int_{t_{0}}^{t_{0}+b-a} L(t_{0},x_{0},x(s), u(s),\omega) ds.
 \end{equation}
 
 \noindent Then using equation \eqref{scaling}, for each $\epsilon >0$ this is equivalent to

\begin{equation} \label{weusethiswhenwe talkaboutF}
\frac{1}{\epsilon} L_{\epsilon (b-a), \epsilon}
(\epsilon t_{0}, x_{0}+\epsilon a f (x_{0}, \widetilde{u}),  \widetilde{u}, \omega ) .
\end{equation}
 
 \noindent The family of measure preserving transformations $\{ \widetilde{\varphi}_{r}  : r \geq 0 \}$ is then defined as $\varphi_{rf(x_{0},\widetilde{u})}$. Thus by Assumption $1$,
 
 \begin{align}
  F_{a,b} \circ \widetilde{\varphi}_{r} &= \inf_{\widetilde{u} \in \Pi_{t_{0},x_{0}+a f (x_{0}, \widetilde{u}), b-a, u}} \int_{t_{0}}^{t_{0}+b-a} L(t_{0},x_{0},x(s) , \varphi_{r f(x_{0},\widetilde{u})}\omega) ds \\
 &=\inf_{\widetilde{u} \in \Pi_{t_{0},x_{0}+a f (x_{0}, \widetilde{u}), b-a, \widetilde{u}}} \int_{t_{0}}^{t_{0}+b-a} L(t_{0},x_{0},x(s) + rf(x_{0},\widetilde{u}) , \omega) ds \\
   &= \inf_{u \in \Pi_{t_{0},x_{0}+(a+r) f (x_{0}, \widetilde{u}), b-a, \widetilde{u}}} \int_{t_{0}}^{t_{0}+b-a} L(t_{0},x_{0},x(s), \omega)ds \\
   &=F_{a+r,b+r}
 \end{align} 
 
\noindent  and so (C1) is satisfied. Sub-additivity (C2) is a consequence of the fact that any path from $x_{0}$ to $ x_{0}+a f (x_{0}, \widetilde{u})$ in time $[0,a]$, and then a path from $ x_{0}+a f (x_{0}, \widetilde{u})$ to $x_{0}+bf (x_{0}, \widetilde{u})$ in time $[a,b]$, can be joined to form a path from $x_{0}$ to $bf (x_{0}, \widetilde{u})$ in time $[0,b]$. The condition (C3) follows from the assumption that $L_{*}(0)$ is finite.\\

We now want to show that it satisfies the conditions  of Proposition \ref{Sub-additive ergodic theorem}.  The upper bound $K(a,b)$ can be taken as $(b-a) L^{*}(K(\widetilde{u}))$ where the $K(\widetilde{u})$ is as in equation \eqref{effectiveboundednessofLagrange}. Hence we only need the following  continuity  lemma.

\begin{lemma} \label{contofF_{a,b}}
%For all $y \in \mathbb{R}^{d}$ and $(a,b) \in Q_{c}$
 For positive real numbers $b_{1},b_{2},a_{1},a_{2}$ such that $|b_{1}-b_{2}|$, $|a_{1}-a_{2}|$ are sufficiently small there exists a constant $C=C(k)$ such that if $ |b_{1}-a_{1}| \wedge |b_{2}-a_{2}|  > k $ 
 
 \begin{equation}
 |F_{a_{1},b_{1}} - F_{a_{2},b_{2}}  \circ \varphi_{y}  | \leq C \big [ |a_{1}-a_{2}| + |b_{1}-b_{2} | + |y| \big ]
 \end{equation}

\end{lemma} 

\begin{proof} 
The proof  occurs through using Lemma \ref{contlemma} in the special case where $\widetilde{u}_{1}=\widetilde{u}_{2}=\widetilde{u}$  and the equation \eqref{weusethiswhenwe talkaboutF}. Doing so we observe that 
\begin{equation} \label{Fcont}
\begin{split} 
|F_{0,b_{1}} \circ \varphi_{y} - F_{0,b_{2}}| &  =  \frac{1}{\epsilon} \big |  L_{\epsilon b_{1},
\epsilon}(\epsilon t_{0}, x_{0}, \epsilon b_{1},  \widetilde{u}, \omega ) -L_{\epsilon b_{2},
\epsilon}(\epsilon t_{0}, x_{0}, \epsilon b_{2},  \widetilde{u}, \omega ) \big |   \\  & \leq  
 b L^{*}(K) |1 / \beta_{*} -1 |  + b \|L\|_{Lip}^{K+\eta(K)} \|H\|^{K+\eta(K)}_{Lip} f^{*}(K)|(\beta^{*} - 1) | +  \frac{L^{*}(M)}{\delta}   |y|
 \end{split}
\end{equation}
provided that $|b_{1}-b_{2}|$ and $y$ are sufficiently small so that the hypothesis of Lemma \ref{contlemma} are satisfied. $\beta^{*}$ and $\beta_{*}$  are as in Lemma \ref{contlemma} and $b=b_{1} \wedge b_{2}$.  Note $\beta^{*}$ and $\beta_{*}$ diverges/becomes indeterminant for small values of $b_{1},b_{2}$ and hence the Lipschitz constant $C(k)$ depends on the parameter $k$.  \\

This proves Lipschitz continuity of the object $F_{a,b} \circ \varphi_{y}$ in the $y$ variable and second parameter. The continuity of the first parameter is proved similar to the analysis in Lemma \ref{contlemma}, in that we adjust a control so that it heads towards the point $x_{0}+af(x_{0},\widetilde{u})$ at a speed $\delta $ in the beginning and from there follow the same trajectory at a faster speed  to compensate. \end{proof}

We  may now utilize the Proposition \ref{Sub-additive ergodic theorem} to deduce on a set $\Omega_{0} \subseteq \Omega $ with $P(\Omega_{0})=1$ there exists an $\widetilde{L}_{\tau} (t_{0},x_{0},\widetilde{u},\omega)$ such that

\begin{equation} \label{linear}
\lim_{\epsilon \rightarrow 0} \epsilon F_{0,\tau /\epsilon}=  \lim_{\epsilon \rightarrow 0}  L_{\tau, \epsilon}(t_{0},x_{0}, \widetilde{u}, \omega) 
= \widetilde{L}_{\tau} (t_{0},x_{0},\widetilde{u},\omega)= \tau \widetilde{L}_{1}(t_{0},x_{0},\widetilde{u},\omega)
\end{equation}

\noindent for all $\omega \in \Omega_{0}$. The linear relationship of $\widetilde{L}$ with respect its argument $\tau$ will become important later. 
%
%We now argue that $\widetilde{L}$ is constant in $\Omega_{0}$.\\ 

Take any $y \in \mathbb{R}^{d}$ and observe that, from \eqref{Fcont},

\begin{equation} 
\lim_{\epsilon \rightarrow 0} \epsilon  \big | F_{0,\tau /\epsilon} \circ \varphi_{y}  - F_{0,\tau /\epsilon} \big | \leq \lim_{\epsilon \rightarrow 0}  \epsilon C |y| =0.
\end{equation}

\noindent Hence it follows that $\widetilde{L}_{\tau}(t_{0},x_{0},\widetilde{u},\omega)$ is $\varphi_{y}$ invariant. Therefore by ergodicity,  $\widetilde{L}_{\tau}(t_{0},x_{0},\widetilde{u},\omega)$ is almost surely constant in $\Omega$.\\

We then extend this to every rational number and take the countable intersection to conclude the following lemma.

\begin{lemma} \label{coneqn}

      There is a function $\widetilde{L}_{\tau}$ such that, for all $(t,x,\widetilde{u}) \in  [0,T] \cap \mathbb{Q} \times \mathbb{Q}^{d} \times \mathbb{Q}^{d} $, we have an event $\Omega^{'} \subset \Omega$ with $P(\Omega^{'})=1$ such that
   \begin{equation}  \label{subadd}
  \lim_{\epsilon \rightarrow 0} L_{\tau, \epsilon}(t,x, \widetilde{u}, \omega) = \widetilde{L}_{\tau}(t,x, \widetilde{u})  \text{ on } \Omega^{'}.
   \end{equation}
   \end{lemma}

   \bigskip
   
  \noindent   For a control $\widetilde{u} \in \mathcal{U}_{\tau}$ we  define the cost function

   \begin{equation} \label{sum}
 J_{\tau}(t,x,\widetilde{u}(s)) = \sum_{i=1}^{N} \widetilde{L}_{\tau}(t_{i},x_{i}, \widetilde{u}_{i})
 \end{equation}

\noindent and value function

 \begin{equation}
 \widetilde{V}_{\tau}(t,x) = \inf_{\widetilde{u} \in \mathcal{U}_{\tau}} J_{\tau}(t,x,\widetilde{u}(s)) .
 \end{equation}
Recalling the linear relationship of $\widetilde{L}_{\tau}$ with respect to $\tau$, we are now a position where we can define the effective Lagrangian for any $\tau >0$.

\begin{equation} \label{deffof1L}
L(t_{0},x_{0}, \widetilde{u}): = \frac{1}{\tau} \widetilde{L}_{\tau} (t_{0},x_{0},\widetilde{u}) =  \widetilde{L}_{\tau} (t_{0},x_{0},\widetilde{u})
\end{equation}

   \begin{lemma} \label{effcontlemma}
   For each $R >  0$, $\widetilde{L}(t,x,\widetilde{u})$ is uniformly continuous on $ [0,T] \cap \mathbb{Q} \times \mathbb{Q}^{d} \times \mathbb{Q}^{d} \cap U^{R}$ and hence can be extended to all real numbers.  
   \end{lemma}

   \begin{proof}

   In this proof we always take $t_{i},x_{i},u_{i}  \in [0,T] \cap \mathbb{Q}\times \mathbb{Q}^{d} \times \mathbb{Q}^{d} \cap U^{R}$.  Employing Lemma \ref{contlemma} and the $\beta_{i}$ therein, for any fixed $\tau$ we will have 
   
   \begin{equation}
   \beta:=\beta_{1}=\beta_{2} = \frac{1}{1 - \delta^{-1} | f(x_{0}, u_{2})- f(x_{0}, u_{2}) |}.
   \end{equation}
   
\noindent  This quantity is also $\tau$ invariant, and thus, by Lemma \ref{contlemma} we immediately obtain a uniform modulus of continuity in the control variable on $U^{R}$ and in $\tau$. \\
   
The temporal and spatial modulus of continuity are inherited from the original Lagrangian, after taking $\epsilon \rightarrow 0$. Specifically, we have
\begin{multline} \label{contofL_alpha,nspace}
 | \widetilde{L}  (t_{1},x_{1},\widetilde{u})  - \widetilde{L} (t_{2},x_{2},\widetilde{u})  | =  | \widetilde{L}_{1} (t_{1},x_{1},\widetilde{u})  - \widetilde{L}_{1} (t_{2},x_{2},\widetilde{u})  | \\ \leq \lim_{\epsilon \rightarrow 0} | L_{1,\epsilon} (t_{2},x_{2},\widetilde{u}, \omega)  - L_{1,\epsilon} (t_{2},x_{2},\widetilde{u}, \varphi_{x_{1}-x_{2}} \omega)  | +\|L\|_{Lip}(|t_{1}-t_{2}|) + m_{L}^{|\widetilde{u}|} (|x_{1}-x_{2}|)   
\end{multline}

\noindent However as the $\widetilde{L}_{\tau}$ is almost surely constant on $\Omega$, then

\begin{equation}
\lim_{\epsilon \rightarrow 0} | L_{1} (t_{2},x_{2},\widetilde{u}, \omega)  - L_{1} (t_{2},x_{2},\widetilde{u}, \varphi_{x_{1}-x_{2}} \omega)  | =0 \text{ a.s.} 
\end{equation}

Therefore 

\begin{equation}
 | \widetilde{L}  (t_{1},x_{1},\widetilde{u})  - \widetilde{L} (t_{2},x_{2},\widetilde{u})  |  \leq \|L\|_{Lip}(|t_{1}-t_{2}|) + m_{L}^{|\widetilde{u}|} (|x_{1}-x_{2}|).
\end{equation}

\end{proof}

   Having extended the effective Lagrangian to all of $\mathbb{R}^{d}$ we can now define the homogenized control problem. For some control $u \in \mathcal{U}$, let $x(\cdot)$ be given by the state dynamics as in \eqref{statedynamics}. We define the cost functional as
   
   \begin{equation}
   \widetilde{J}(t,x,u) := \int_{t}^{T} \widetilde{L} (s,x(s),u(s)) ds
   \end{equation}
   
 \noindent  and the value function
   
   \begin{equation}
   \widetilde{V}(t,x):= \inf_{u \in \mathcal{U}} \widetilde{J}(t,x,u).
   \end{equation}

\subsection{Coercivity of the effective Lagrangian}

In this section we prove that the effective Lagrangian satisfies the same coercive bounds as the original Lagrangian. 

\begin{lemma} \label{Coercivity Lemma}
For all $t_{0},x_{0},u \in \mathbb{R}^{d}$ and $\epsilon, \tau >0$ we have that

\begin{equation}
L_{\tau, \epsilon} (t_{0},x_{0},u) \geq \tau L_{*}(u)
\end{equation}

\end{lemma}

\begin{proof}

Without loss of generality, take $\tau=1$ and suppose that  $\tau_{1},\tau_{2} \geq 0$  such that $\tau_{1}+\tau_{2}=1$.  \\\\ Suppose that $u_{1}$ and $u_{2}$ are such that

\begin{equation}
\tau_{1} f(x_{0},u_{1}) + \tau_{2} f(x_{0},u_{2})  = f(x_{0},u)
\end{equation}

\noindent then the essence of the proof is showing that this implies that

\begin{equation}
\tau_{1} L_{*}(u_{1}) + \tau_{2} L_{*}(u_{2})  \geq L_{*} (u).
\end{equation}

The argument is as follows. We have from Assumption $8$ that

\begin{equation}
\tau_{1} L_{*}(u_{1}) - \tau_{1} L_{*}(u) \geq  \tau_{1} \Theta (u_{2}) \cdot \big [ f(x,u_{1})- f(x,u) \big ] 
\end{equation}

\noindent and

\begin{equation}
\tau_{2} L_{*}(u_{2}) - \tau_{2} L_{*}(u) \geq  \tau_{2} \Theta (u_{2})  \cdot \big [ f(x,u_{2})- f(x,u) \big ] .
\end{equation}

\noindent Adding these gives the inequality

\begin{equation}
\tau_{1} L_{*}(u_{1}) + \tau_{2} L_{*}(u_{2})  \geq L_{*} (u).
\end{equation}

\noindent Iterating and passing to limits we can conclude for any $u \in \Pi_{t_{0},x_{0},\tau,\widetilde{u}}$,

\begin{equation}
\int_{t_{0}}^{t_{0}+\tau} L \left( t_{0},x_{0}, \frac{x(s)}{\epsilon}, u(s),\omega \right) ds  \geq \int_{t_{0}}^{t_{0}+\tau} L_{*} (u(s)) ds \geq L_{*}(\widetilde{u}).
\end{equation}
\end{proof}

By taking $\epsilon \rightarrow 0$ we can immediately conclude from this Lemma that $\widetilde{L}_{\tau}(t,x,u) \geq \tau L_{*}(u)$ for every $\tau >0$.
This lemma gives us precisely the same estimates on $\widetilde{L} $, $L_{\tau,\epsilon}$ as $L$, and so we can prove the effective boundedness for the homogenized control problems.

\begin{lemma} \label{effectiveboundednessofvaluelemma2}
There exists an $K >0$ such that for all $t_{0} \in [0,T] ,x_{0} \in \mathbb{R}^{d}$.

\begin{equation} \label{effectiveboundednessofvalue2}
\widetilde{V}^{K}(t_{0},x_{0})=\widetilde{V}(t_{0},x_{0}),
\end{equation}

and

\begin{equation}
\widetilde{V}^{K}_{\tau}(t_{0},x_{0})=\widetilde{V}_{\tau}(t_{0},x_{0}).
\end{equation}

\end{lemma} 

Let us analyse the example with $L_{*}(u)=|u|^{\beta} $ and

\begin{equation}
f(x,u) = \frac{ C u} {\sqrt{|u|^{2}+1} }
\end{equation}

We can transform this into a calculus of variations problem on the open ball $B(0,C)$. In particular for a given  control $u(s) \in \Pi_{t_{0}, x_{0},\tau,\widetilde{u}}$ we may rewrite the Lagrangian using the inverse formula \eqref{inverse}

\begin{equation} 
\int_{t_{0}}^{t_{0}+\tau} L_{*}(u(s)) ds = \int_{t_{0}}^{t_{0} + \tau} |u(s)|^{\beta} ds  = \int_{t_{0}}^{t_{0} + \tau} \frac{|v(s)|^{\beta} }{\Big ( \sqrt{1-\frac{v(s)^{2}}{C^{2}} }\Big )^{\beta}}ds.
\end{equation}

It is elementary to verify that the integrand on the RHS is a convex function, provided that $\beta \geq 1$.  Hence we can conclude that

\begin{equation}
\int_{t_{0}}^{t_{0}+\tau} L_{*}(u(s)) ds  \geq  \tau L_{*}(\widetilde{u}). 
\end{equation}

In order to understand the nature of this $\Theta$, consider the case where $f$ and $L_{*}(\cdot)$ is continuously differentiable and that the Jacobian of $f$ is invertible at every point. Then similar to lemmas in classic convex analysis we can prove that the function $\Theta(\cdot)$ is the gradient of $L_{*}(\cdot)$ multiplied by the inverse of the Jacobian. Given the conclusion that

\begin{equation}
(1-\tau) L_{*}(u_{1}) + \tau L_{*}(u_{2}) \geq L_{*} (\widetilde{u}) 
\end{equation}

for $\tau \in [0,1]$ and $u_{1},u_{2}$ such that $(1-\tau) f( x_{0}, u_{1}) + \tau f(x_{0},u_{2}) = f (x_{0},u)$, we can make a brief calculation to deduce that 

%We use the Taylor series expansion
%
%\begin{align}
%u  & = u_{1} +  J^{-1}(x_{0}, u_{1}) \big [ f(x_{0},u) - f(x_{0}, u_{1}) \big ] + \mathcal{O}  (u_{1}-u_{2})^{2} \\
%& =u_{1} +  \tau J^{-1}(x_{0}, u_{1}) \big [ f(x_{0},u_{2}) - f(x_{0}, u_{1}) \big ] + \mathcal{O}  (u_{1}-u_{2})^{2}.
%\end{align}
%
%
%for $\tau \in [0,1]$ and $u_{1},u_{2}$ such that $(1-\tau) f( x_{0}, u_{1}) + \tau f(x_{0},u_{2}) = f (x_{0},u)$ . We can deduce that

\begin{equation}
L_{*}(u_{2}) - L_{*}(u_{1}) \geq \frac{ L_{*}(u_{1}) - L_{*}(u_{1} + \tau J^{-1} (x_{0}, u_{1}) [ f(x_{0},u_{1} ) - f(x_{0},u_{2}) +o(\tau)] )}{\tau} 
\end{equation}

Taking $\tau \rightarrow 0$ implies that

\begin{equation}
L_{*}(u_{2}) - L_{*}(u_{1})  \geq \nabla L_{*}(u_{1}) \cdot [ J^{-1}(x_{0},u_{1}) ( f(x_{0},u_{2}) - f(x_{0},u_{1}) ]
\end{equation}

\noindent Thus we see that this function $\Theta$ is more easily classified then it make at first seem.

In the generality we are interested in, Coercivity of the Homogenized Lagrangian cannot be taken for granted. Suppose that $f$ was not injective, so for some $\widetilde{u}_{1},\widetilde{u}_{2}$ we had that

\begin{equation} \label{non-injective}
f(x_{0},\widetilde{u}_{1}) = f(x_{0},\widetilde{u}_{2}).
\end{equation}

Then $\Pi_{t_{0}, x_{0},\tau,\widetilde{u}_{1}} = \Pi_{t_{0}, x_{0},\tau,\widetilde{u}_{2}}$. 
From the definition of $\widetilde{L}$ this implies that

\begin{equation}
\widetilde{L}(t_{0},x_{0},\widetilde{u}_{1})= \widetilde{L}(t_{0},x_{0},\widetilde{u}_{2}).
\end{equation}

We know that $L_{*}(u) \rightarrow \infty$. 
Hence if equation \eqref{non-injective} was true and $|\widetilde{u}_{1}| \gg |\widetilde{u}_{2}|$ then it would follow

\begin{equation}
L_{*}(\widetilde{u}_{1}) > \widetilde{L}(t_{0},x_{0},\widetilde{u}_{2}) = \widetilde{L}(t_{0},x_{0},\widetilde{u}_{1}).
\end{equation} 

It is then  clear that Lemma \ref{Coercivity Lemma} must be false. Conversely the second statement of Assumption \ref{Growthassumption} implies something similar to injectivity.  Namely if we Assumption \ref{Growthassumption} and   $f(x_{0},\widetilde{u}_{1}) = f(x_{0},\widetilde{u}_{2})$ then this implies

\begin{equation}
L_{*}(\widetilde{u}_{1}) =  L_{*}(\widetilde{u}_{2}).
\end{equation} 

If we have chosen $L_{*}(\widetilde{u}_{1})$ so that it is a strictly increasing function, and $|\widetilde{u}_{1}| \neq |\widetilde{u}_{2}|$ this is a impossible. Hence we see that encoded into Assumption  \ref{Growthassumption} is that non-injectivity of $f$ is heavily restricted.

   \subsection{Convergence of the Value functions}
    
  In this subsection we prove Theorem \ref{maintheorem1}.

 \begin{lemma} \label{epsilonlemma}
 For every $t \in [0,T], x \in \mathbb{R}^{d}$ and almost all $\omega \in \Omega$
  
  \begin{equation}
  V_{\tau,\epsilon} (t,x, \omega) \rightarrow  V_{\epsilon} (t,x, \omega)
  \end{equation}
   uniformly in $\epsilon$ as  $\tau  \rightarrow 0$.
 \end{lemma}

 \begin{proof}
 
Similar to the preceding sections,  the argument is based on imitating the trajectory and showing that the difference in cost is small.\\

Let $\kappa > 0 $ be arbitrary, take a $u \in \mathcal{U}$ such that 

\begin{equation}
J_{\epsilon}(t,x,u,\omega) \leq V_{\epsilon} (t,x,\omega) + \kappa.
\end{equation}

Using lemma \ref{effectiveboundednessofvaluelemma} we can without of generality assume that $\|u\|_{\infty} \leq K$. 
Define an associated approximation $\widetilde{u} $ in $\mathcal{U}_{\tau}$ by iterative construction. Set

 \begin{equation}
 t_{1}= t_{0} + \tau  \wedge T
 \end{equation}
 
 \noindent To keep notation consistent we will set $x_{0}:=x(t_{0})$ We will show that there is a $\widetilde{u}_{1}\in U^{K}$ such that

 \begin{equation}
 f(x_{0},\widetilde{u}_{1})=\frac{x(t_{1})-x}{\tau}.
 \end{equation}
 We can, without any loss of generality, suppose that the control is a step function taking values $u_{k}$ at times $  [\bar{t}_{k}, \bar{t}_{k+1}) $  (see \cite{royden1988real} Proposition 22, Page 69). Then 
 
 \begin{equation}
 x(t_{1}) - x(t_{0}) = \sum_{k=1}^{N} \int_{\bar{t}_{k}}^{\bar{t}_{k+1}}  f(x(s),u_{k}) ds.
 \end{equation}
 
\noindent So, stipulating that $\tau >0$ is sufficiently small so that $\eta( K) \leq \tau f^{*}(K)$, we can use Assumption \ref{implicitfunctionassumption} to deduce that there exist $\bar{u}_{k}$ such that $f(x(s),u_{k})=f(x_{0},\bar{u}_{k})$ at each $[\bar{t}_{k} , \bar{t}_{k+1}]$. 
 
 \begin{equation}
 x(t_{1}) - x(t_{0}) = \sum_{k=1}^{N} \int_{\bar{t}_{k}}^{\bar{t}_{k+1}}  f(x_{0},\bar{u}_{k})ds = \sum_{k=1}^{N} [\bar{t}_{k+1}-\bar{t}_{k}] f(x_{0},u_{k}).
 \end{equation}
 
\noindent Divide both sides by $\tau $ and use that fact that the sum of $\bar{t}_{k}-\bar{t}_{k+1}$ is $\tau$ and $f(x_{0},U^{K})$ is a convex set to derive the existence of a $\widetilde{u}_{1}$ such that 

\begin{equation}
 f(x_{0},\widetilde{u}_{1})=\frac{x(t_{1})-x}{\tau}.
 \end{equation}
 
% 
% In the general case, we take a step function $\{ u_{k} \}_{k=1}^{N}$ such that
% 
% \begin{equation}
% \Big | \int_{t_{0}}^{t_{1}} f(x_{0},u(s)) - \sum_{k=1}^{N} \big[ \bar{t}_{k+1}-\bar{t}_{k} \big ] \Big | \leq \widetilde{R}
% \end{equation}
% 
%
  We repeat this procedure and thus define

 \begin{equation} \label{howtogetcontrol1}
   t_{i}:=t_{i-1}+\tau \wedge T
 \end{equation}
 
Similarly there exists $\widetilde{u}_{i}$ such that
\begin{equation} \label{howto getcontrol2}
  f(x_{i},\widetilde{u}_{i})=\frac{x(t_{i})-x_{i-1}}{\tau }
 \end{equation}

\noindent and we define $x_{i}=x(t_{i})$. We terminate on the first $N$ such that $t_{N}=T$.  These $\widetilde{u}_{i}$ then define a $\widetilde{u} $ in $\mathcal{U}_{\tau}$.\\ 
  The approximate effective Lagrangian was defined to be the infinmum of all paths starting at $x_{i}$ and ending at $x_{i+1}$ according to the frozen dynamics $f(x_{i},u)$. Therefore the approximate effective Lagrangian will be smaller than the given control, $u$, over $[t_{i}, t_{i+1})$, except perhaps some small amount occurring from the difference in the dynamics if we trace out the same control, and a difference in a time argument and space argument. This is precisely what occurs in Lemma \ref{whythisworks}. The inequality deduced there implies  (in light of the fact that $t_{i+1}-t_{i} \leq \tau$)
  \begin{multline}
L_{\tau, \epsilon} (t_{i},x_{i},\widetilde{u}_{i},\omega)  - \tau^{2}  \big [ \|L\|_{Lip} + f^{*} (K)  \|L\|_{Lip}^{K+\eta(K)} \|H\|_{Lip}^{K+\eta(K)}  \big ] - \tau m_{L}^{ K+\eta(K)}(\tau f^{*}(K))  \\ \leq \int_{t_{i}}^{t_{i+1}}  L(s,x(s),\frac{x(s)}{\epsilon}, u(s), \omega) ds.
   \end{multline}
 \
  
 \noindent   Summing over $i$  we conclude
  \begin{multline} \label{ConvofV}
J_{\tau, \epsilon}(t_{0},x_{0}, \widetilde{u}, \omega)- T \tau \big [ \|L\|_{Lip} + f^{*} (K)  \|L\|_{Lip}^{K+\eta(K)} \|H\|_{Lip}^{K+\eta(K)}  \big ] - T m_{L}^{ K+\eta(K)}(\tau f^{*}(K) ) \\   \leq V_{\epsilon}(t_{0},x_{0},\omega)+\kappa.
  \end{multline}

   Conversely for a given $\kappa>0$, suppose $\widetilde{u}=\{\widetilde{u}_{i} \}_{i=1}^{N} \in \mathcal{U}_{\tau}$ is such that
   
   \begin{equation}
   J_{\tau,\epsilon}(t,x,\widetilde{u},\omega) \leq V_{\tau, \epsilon} (t,x,\omega) +\frac{\kappa}{2}.
   \end{equation}
   
 \noindent  We know by Lemma \ref{effectiveboundednessofvaluelemma} that we can without loss of generality assume that $|\widetilde{u}_{i}| \leq K$ for every $i$, uniformly in $\tau, \epsilon$. \\\\
\indent Let $u_{i}$ be a sequence of controls such that, for each $i \leq N$, 
     
     \begin{equation}
    \int_{t_{i}}^{t_{i+1}}  L\left(t_{i},x_{i}, \frac{x(s)}{\epsilon} ,u_{i}(s),\omega\right)ds  \leq  \widetilde{L}_{\tau, \epsilon} (t_{i},x_{i},\widetilde{u_{i}}) + \frac{\kappa}{2} (t_{i+1}-t_{i}).
     \end{equation}
We know by Lemma \ref{effectiveboundednessofvaluelemma} that for every $\widetilde{u}_{i}$ there is a $\widetilde{K}(\widetilde{u}_{i})$  such that we may assume $\|u_{i}\|_{\infty} \leq \widetilde{K}(\widetilde{u}_{i})$. However as $\sup_{i \leq N} | \widetilde{u}_{i}| \leq K$ then we may strengthen this to $\sup_{i \leq N} \|u_{i}\|_{\infty} \leq \widetilde{K} (K)=:\widehat{K}$.   

Hence
   \begin{multline} \label{othersideofinequality}
 \widetilde{L}_{\tau,\epsilon} (x_{i},u_{i},\omega)  -  \tau^{2} \big [ \|L\|_{Lip} + f^{*} (K)  \|L\|_{Lip}^{\widehat{K}+\eta(\widehat{K})} \|H\|_{Lip}^{\widehat{K}+\eta(\widehat{K})}  \big ] - \tau m_{L}^{ \widehat{K}+\eta(\widehat{K})}(\tau f^{*}(\widehat{K}) +\eta(\widehat{K})) + \frac{\kappa}{2}  \\ \geq \int_{t_{i}}^{t_{i+1}} L\left(s,x(s),\frac{x(s)}{\epsilon}, u(s), \omega\right)ds.
   \end{multline}

\noindent As $\kappa > 0$ was arbitrary, summing up over $i$ we get the converse of \eqref{ConvofV}. Thus taking  $K^{*}:= \widehat{K} \wedge K$ such that
\begin{align} \label{ThisV} 
& \big | V_{\tau, \epsilon} (t,x, \omega) - V_{\epsilon} (t,x ,\omega ) | \\
& \leq T \tau \big [ \|L\|_{Lip} + f^{*} (K^{*})  \|L\|_{Lip}^{K^{*}+\eta(K^{*})} \|H\|_{Lip}^{K^{*}+\eta(K^{*})}  \big ] + T m_{L}^{ K^{*}+\eta(K^{*})}(\tau f^{*}(K^{*}) )+ \kappa.
\end{align}
 Observing the the RHS is independent of $\epsilon$ completes the proof.
   \end{proof}

We now show that the discrete homogenized control problem converges to the continuous homogenized problem as $\tau \rightarrow 0$.

 \begin{lemma} \label{The lemma that uses step functions} 

For each $t \in [0,T], x \in \mathbb{R}^{d}$
   \begin{equation}
\lim_{\tau \rightarrow 0} \widetilde{V}_{\tau} (t,x)  = \widetilde{V} (t,x) .
     \end{equation}
     
 \end{lemma}

 \begin{proof}  Let $\widetilde{u} = \{ \widetilde{u}_{i} \}_{i=1}^{N} \in \mathcal{U}_{\tau}$, with time associated intervals $[t_{i},t_{i+1})$ and spatial points $\widetilde{x}_{i}$. As usual we can use Lemma \ref{effectiveboundednessofvaluelemma} to assume that  for some $K$ we have $|\widetilde{u}_{i}| \leq K$ for every $i$. We can define a $u \in \mathcal{S}$ such that $f(x_{i},\widetilde{u}_{i}) =f(x(s),u(s)) $ on $ s \in[t_{i},t_{i+1})$. Using Assumption \ref{implicitfunctionassumption} this will satisfy $|\widetilde{u_{i}}-u(s)| \leq \tau f^{*}(K) \|H\|^{K}_{Lip} $.
Then we have that

 \begin{equation} \label{partitionofindia}
 \widetilde{J}(t,x,u) = \int_{t}^{T} \widetilde{L} (s,x(s),u(s))=  \sum_{i=1}^{N} \int_{t_{i}}^{t_{i+1}} \widetilde{L} (s,x(s),u(s))ds.
 \end{equation}
Because 
\begin{equation} 
|\widetilde{L} (s,x(s),u(s)) - \widetilde{L}(t_{i},x_{i},\widetilde{u}_{i}) | \leq \tau \|\widetilde{L}\|_{Lip}^{K+\eta(K)}  \|H\|_{Lip}^{K+\eta(K)} + \tau \|L\|_{Lip} + m_{L}^{K+\eta(K)} (\tau f^{*}(K )) 
\end{equation}
  and by definition of $\widetilde{L}$ in \eqref{deffof1L} we have

  \begin{equation}
 \int_{t_{i}}^{t_{i+1}} \widetilde{L} (t_{i},x_{i},\widetilde{u}_{i})ds = \tau  \widetilde{L} (t_{i},x_{i},\widetilde{u}_{i})ds
 =  \widetilde{L}_{\tau} (t_{i},x_{i},\widetilde{u}_{i}).
 \end{equation}
 Then,
 \begin{equation}
 |\widetilde{J}_{\tau} (t,x,\widetilde{u}) - \widetilde{J}(t,x,u)| \leq \tau \|\widetilde{L}\|_{Lip}^{K+\eta(K)}  \|H\|_{Lip}^{K+\eta(K)} + \tau \|L\|_{Lip} + m_{L}^{K+\eta(K)} (\tau f^{*}(K)  )
 \end{equation}
In particular if given a $\kappa > 0$ and a control $\widetilde{u}$ such that 
 
 \begin{equation}
 \widetilde{J}_{\tau} (t,x,u) \leq \widetilde{V}_{\tau}(t,x) + \kappa
 \end{equation}
then this proves that
 
 \begin{equation}  \label{Lemmathatusesstepfunctiononesideofineq1}
 \widetilde{V}_{\tau}(t,x)  +  \tau \| \widetilde{L}\|_{Lip}^{K+\eta(K)}  \|H\|_{Lip}^{K+\eta(K)} + \tau \|L\|_{Lip} + m_{L}^{K+\eta(K)} (\tau f^{*}(K ))  + \kappa \geq  \widetilde{V}(t,x).
 \end{equation}

Conversely, fix a control $u \in \mathcal{U}$. Without loss of generality we may assume that $u \in \mathcal{S}$. We now define a $\widetilde{u} \in \mathcal{U}_{\tau}$ precisely as in the previous lemma through equations \eqref{howtogetcontrol1}, \eqref{howto getcontrol2} and denote the sequence of controls in $\mathcal{U}$  as $\widetilde{u}_{i}$ and it's time partitions as $[\widetilde{t}_{j},\widetilde{t}_{j+1}]$. If we have for some $j, i$ that $ [\widetilde{t}_{j},\widetilde{t}_{j+1}] \subset [t_{i},t_{i+1}]$, then it follows that
 $|\widetilde{u_{i}}-u(s)| \leq  \tau f(K) \|H\|^{K}_{Lip} $ as $f(x_{i},u_{i})=f(x(s),u(s))$ in this time period, which implies that during this time
 
 \begin{equation} 
 |\widetilde{L} (s,x(s),u(s)) - \widetilde{L}(t_{i},x_{i},u_{i}) | \leq \tau \| \widetilde{L}\|_{Lip}^{K}  \big \|H\|_{Lip}^{K} +\tau \|L\|_{Lip} + m_{L}^{K+\eta(K)} (\tau f^{*}(K) )
 \end{equation}
 
 \noindent We want to argue that the vast majority of $j$ falls within such an $i$. 
 
 Let $N$ be the number of partitions of the control $u$. Then observe that all but $N$ of the $j$ must satisfy $[\widetilde{t}_{j},\widetilde{t}_{j+1}] \subset [t_{i},t_{i+1}]$ and on each $j$ we have that $\widetilde{t}_{j+1}-\widetilde{t}_{j} \leq \tau$. Then
 \begin{equation} \label{stepcostbound}
 |\widetilde{J}(t,x,u) - \widetilde{J}_{\tau} (t,x,\widetilde{u})|  \leq T \tau \big ( \| \widetilde{L}\|_{Lip}^{K}  \big \|H\|_{Lip}^{K} + \|L\|_{Lip} +N L^{*}(K) \big ) + m_{L}^{K+\eta(K)} (\tau f^{*}(K)) 
 \end{equation}
 
 \noindent In particular if we have for a given $\kappa >0$ that 
 
 \begin{equation}
 \widetilde{J}(t,x,u) \leq \widetilde{V}(t,x) +\kappa.
 \end{equation}
 
\noindent Then we deduce that
 \begin{equation} \label{Lemmathatusesstepfunctiononesideofineq2}
 \widetilde{V}(t,x) + T \tau \big ( \| \widetilde{L}\|_{Lip}^{K}  \big \|H\|_{Lip}^{K} + \|L\|_{Lip} +N L^{*}(K) \big ) + m_{L}^{K+\eta(K)} (\tau f^{*}(K)) 
+ \kappa \geq \widetilde{V}_{\tau}(t,x).
 \end{equation}

 \noindent We now take $\tau \rightarrow 0$, as $\kappa$ was arbitrary, this concludes the proof. 
 \end{proof}

   \begin{remark}
 In the proofs of Lemma's \ref{epsilonlemma} and \ref{The lemma that uses step functions}, the terminal cost function plays no role. The discrete trajectory and continuous are always such that $x_{N}=x(T)$. This will also be true in the Lemmas  \ref{StepfunctionLemma} and \ref{effectiveboundednessofvaluelemma} which will be proved later.
  \end{remark}

We can now combine this to prove the main result.

%\bigskip

\begin{proof}[Proof of Theorem 1.1] For all $t \in [0,T]$, $x \in \mathbb{R}^{d} $ and almost all $\omega \in \Omega$ we have from the triangle inequality,
 \begin{align} \label{triineq}
&  \big | V_{\epsilon}(t,x, \omega) - \widetilde{V} (t,x) \big | \\
%&\quad\leq \big | V_{\epsilon} (t,x, \omega) -V_{\tau, \epsilon} (t,x, \omega) \big | + \big | \widetilde{V} (t,x)-  V_{\tau, \epsilon} (t,x) \big | \\
&\quad  \leq | V_{\epsilon} (t,x,\omega) -V_{\tau, \epsilon} (t,x) \big | +  \big | V_{\tau, \epsilon} (t,x)-  \widetilde{V}_{\tau}(t,x) \big |+\big |\widetilde{V}_{\tau} (t,x) - \widetilde{V} (t,x)\big |. 
\end{align}

For the second term, 
 \begin{equation}
  \lim_{\epsilon \rightarrow 0} \big | \widetilde{V}_{\tau, \epsilon} (t,x)-  \widetilde{V}_{\tau}(t,x) \big | =0
 \end{equation}
\noindent  as $L_{\tau, \epsilon}(t_0,x_0,\widetilde{u},\omega) \rightarrow L_{\tau}(t_0,x_0,\widetilde{u})$ for each $(t_0,x_0, \widetilde{u})$ almost surely and $L_{\tau}$ is uniformly continuous.

 Taking $\epsilon \rightarrow 0$ and using Lemma \ref{epsilonlemma}, we have then that, on the event $\Omega^{'}$,
\begin{equation} 
 \limsup_{\epsilon \rightarrow 0} \big | V_{\epsilon}(t,x) - \widetilde{V} (t,x) \big |
  \leq \big | \widetilde{V} (t,x)-  \widetilde{V}_{\tau} (t,x) \big |+\limsup_{\epsilon \rightarrow 0} |V_{\epsilon}(t,x)-V_{\tau, \epsilon}(t,x,\omega)|.
 \end{equation}
 
 \noindent As $\tau$ is arbitrary, we have the result by Lemmas \ref{epsilonlemma}, \ref{The lemma that uses step functions}. Uniform convergence  on compact sets follows from equicontinunity  of $V_{\epsilon}(t,x)$.  \end{proof}

 \subsection{Homogenization of the Hamiltonian Jacobi equation }
  In this subsection we detail the equivalence of homogenizing our control problem, to homogenizing  the associated class of Hamilton Jacobi equation. This thereby proves Theorem \ref{maintheorem2}. \\\\
  
If we define the Hamiltonians
\begin{align}
\mathcal{H} \left(t,x,\frac{x(s)}{\epsilon}, p, \omega\right)
&:=\sup_{v \in \mathbb{R}^{d}} \left\{ - f(x,v) \cdot p - L\left( t,x,\frac{x}{\epsilon}, v, \omega \right) \right\},\\
 \widetilde{\mathcal{H}} (t,x,p)&:= \sup_{v \in \mathbb{R}^{d}} \left\{ - f(x,v) \cdot p - \widetilde{L}(t,x,v) \right\},
\end{align}
then it is a well known fact that the value functions $V_{\epsilon}$ and $\widetilde{V}$ are the unique viscosity solutions to the following Hamilton-Jacobi equations: \footnote{We refer the reader to \cite{fleming2006controlled} page 104 for details. We mention that the continunity of the Lagrangians, state dynamics and effective compactness of the control space are essential ingredients.} 
\begin{align}
-\frac{\partial V_{\epsilon} }{\partial t} +\mathcal{H} \left(t,x,\frac{x(s)}{\epsilon}, D_{x} V_{\epsilon}, \omega\right)&=0, \:\:\:\: V_{\epsilon}(T,x)=\psi(x)\\
-\frac{\partial \widetilde{V}}{\partial t} +\widetilde{\mathcal{H}} (t,x, D_{x} \widetilde{V})&=0 , \:\:\:\: \widetilde{V}(T,x)=\psi(x)
\end{align}

Thus, having proved that $V_{\epsilon} \rightarrow \widetilde{V}$ uniform on compact subsets, Theorem \ref{maintheorem2} is immediate.

\section{Proof of technical lemmas}\label{Tech}

\subsection{Approximation by piecewise constant dynamics} 

\textit{Proof of Lemma \ref{StepfunctionLemma}.} \\ \textbf{Step 1}:  Without loss of generality take $t=0$. Our first step is to prove that we can approximate the value function using controls that are step functions. Fix a $(t,x) \in [0,T] \times \mathbb{R}^{d}$ and a control $u \in \mathcal{U}$. We know from standard analysis (e.g see \cite{royden1988real}, Proposition 22, page 69.) that for any given $\kappa >0$ we can take $\bar{u}$ to be a step function such that, % for every $s \in [t_{k}, t_{k+1}]$

\begin{equation} 
|\bar{u}(s)-u(s)| \leq \kappa
\end{equation}
except perhaps on a set of measure less then $ \kappa$ and  $  \|\bar{u}\|_{\infty} \leq \|u\|_{\infty}$. Denoting $\bar{x}(\cdot)$ as the trajectory of the control $\bar{u}(\cdot)$ we have that 
\begin{align}
\big | J(0,x,u)-J(0,x,\bar{u}) \big | &\leq \int_{0}^{\bar{T}} |\bar{L}(s,x(s),u(s))-\bar{L}(s,\bar{x}(s),\bar{u}(s)) |ds\\
 \label{ineqofcontrol} 
&\leq  \int_{0}^{\bar{T}} m_{\bar{L}}^{\|u\|_{\infty}}  (|\bar{x}(s)- x(s)| )ds + \kappa \bar{T} \|L\|^{\|u\|_{\infty}}_{Lip} + \kappa  L^{*}(\|u\|_{\infty}).
\end{align}
We then apply Gronwall's inequality on $|\bar{x}(s)- x(s)|$. Specifically for each $t \in [0, \bar{T}]$,
\begin{align}
|\bar{x}(s) -x(s) | & \leq \int_{0}^{\bar{T}} |\bar{f}(x(r),u(r)) - \bar{f}(\bar{x}(r), \bar{u}(r) ) |dr \\ 
& \leq \|f\|^{\|u\|_{\infty}}_{Lip} \int_{0}^{\bar{T}}  (|\bar{x}(r)-x(r)| + |\bar{u}(r)-u(r)|) dr \\
& \leq  \|f\|^{\|u\|_{\infty}}_{Lip} \int_{0}^{\bar{T}}  (|\bar{x}(r)-x(r)| ) dr + \kappa (T+1).
\end{align}
Therefore $|\bar{x}(s)-x(s) |  \leq  \kappa (\bar{T}+1)\text{exp}(\|f\|^{\|u\|_{\infty}}_{Lip}(T+1)) $ for all $t$ substituting this into  equation \eqref{ineqofcontrol}, we deduce that

\begin{equation}
\big | J(0,x,u)-J(0,x,\bar{u}) \big |  \leq\bar{T} m_{L}^{\|u\|_{\infty}} \big ( \kappa | (\bar{T}+1)\text{exp}(\|f\|^{\|u\|_{\infty}}_{Lip}(\bar{T}+1)) | \big ) +  \kappa \bar{T} \|L\|^{\|u\|_{\infty}}_{Lip} + \kappa  L^{*}(\|u\|_{\infty}).
\end{equation} 

\noindent This completes step one.\\

\noindent\textbf{Step 2}: In this step we use Assumption \ref{implicitfunctionassumption} to show that step one can be converted into the statement about step functions. Let $\bar{u}(s)$ be as above with $\bar{u}(s)= \bar{u}_{i}$ on $[t_{i}, t_{i+1})$. For $ i  \leq N$, define the quantity $\widetilde{v_{i}}$ as

\begin{equation}
 \widetilde{v}_{i} =\frac{x(t_{i+1}) - x(t_{i})}{t_{i+1}-t_{i}} =  \frac{1}{t_{i+1}-t_{i}} \int_{t_{i}}^{t_{i+1}} \bar{f}(x(s),u(s))ds.
\end{equation}
Note that this is the average of the velocity produced by the original control, $u$. Without loss of generality we can take the  time interval of the step function $\bar{u}(s)$ so small so that that $\sup_{i \leq N} | t_{i+1}-t_{i}| f^{*}(\|u\|_{\infty}) \leq \eta(\|u\|_{\infty})$. Then using Assumption \ref{implicitfunctionassumption} and \ref{convexassumption} there exists a control $\widetilde{u}$ such that
 
 \begin{equation}
 \bar{f}(\widetilde{x}(s), \widetilde{u}(s)) = \widetilde{v_{i}}
 \end{equation}
for $s \in [t_{i}, t_{i+1})$, where $\widetilde{x}(s)$ is piecewise linear trajectory from $x(t_{i})$ to $x(t_{i+1})$. We then extend this control to all $[0,\bar{T}]$. Observe in particular that $x(t_{i})=\widetilde{x}(t_{i})$ and thus $\widetilde{x}$ will also satisfy the same terminal conditions (and have identical terminal cost). Then we have that
 
 \begin{align}
|x(t) - \widetilde{x}(t) |  & \leq |\bar{x}(t) - x(t)| + |\widetilde{x}(t)-\bar{x}(t)|  \\
& \leq \|f\|^{\|u\|_{\infty}}_{Lip} \int_{0}^{t} \big (|\bar{x}(s)- \widetilde{x}(s)| +|\bar{u}(s)- \widetilde{u}(s)| \big ) ds + \kappa (\bar{T}+1)  \text{exp}(\|f\|^{\|u\|_{\infty}}_{Lip}(\bar{T}+1)\big ).
 \end{align}
 
 %\begin{equation}
%\leq K_{\| u \|_{\infty}} \int_{0}^{t} |x^{1}(s)- \widetilde{x}(s)| +m_{g,\|u\|_{\infty}} | u^{1}(s)-\widetilde{u}(s)|ds
% \end{equation}
 
 \noindent By Gronwall's inequality
 \begin{equation} \label{boundontraject}
 |\bar{x}(t)- \widetilde{x}(t) | \leq  \Big (  \int_{0}^{t} |\bar{u}(s)-\widetilde{u}(s)|ds+ \kappa (\bar{T}+1) \Big ) \exp(2 \|f\|^{\|u\|_{\infty}}_{Lip}  (\bar{T}+1)).
 \end{equation}
Thus we want to acquire a bound on $|\bar{u}(s)-\widetilde{u}(s)|$. But, because $\bar{u}$ is a step function, then we have that, except perhaps on a set of measure less then $\kappa$ (using $|\bar{u}(s) - u(s)| \leq \kappa$),

\begin{equation}
 \big | \bar{f} (x(s),u(s))  - f(\bar{x}(t_{i}), u_{i}) \big | \leq \|f\|_{Lip}^{\|u\|_{\infty}}  |t_{i+1}-t_{i} | (f^{*}(\|u\|_{\infty}) + \kappa)
\end{equation}

This implies that, if we integrate over $[t_{i}, t_{i+1})$ we get 

\begin{equation}
|  \bar{f}(\bar{x}(t_{i}), u_{i}) - \widetilde{v}_{i}  | \leq \|f\|^{\|u\|_{\infty}}_{Lip} ( |t_{i+1}-t_{i} | f^{*}(\|u\|_{\infty} )+\kappa )+ \kappa .
\end{equation}
Thus,
\begin{align}
|   \bar{f} (x(s), u_{i}) - \widetilde{v}_{i} | & \leq \|f\|_{Lip}^{\|u\|_{\infty}}|t_{i+1}-t_{i}| f^{*}(\|u\|_{\infty} )+ |  f(\bar{x}(t_{i}), u_{i}) - \widetilde{v}_{i}  | \\
& \leq  2\|f\|_{Lip}^{\|u\|_{\infty}}|t_{i+1}-t_{i}| f^{*}(\|u\|_{\infty} ) +  \kappa (\|f\|^{\|u\|_{\infty}}_{Lip} +1).
\end{align}
Then we can utilize Assumption \ref{implicitfunctionassumption} to deduce that
\begin{equation}
|\bar{u}(s)-\widetilde{u}(s) | \leq \|H\|_{Lip}^{\|u\|_{\infty}}  \big [ 2\|f\|_{Lip}^{\|u\|_{\infty}}|t_{i+1}-t_{i}| f^{*}(\|u\|_{\infty} )  + \kappa (\|f\|^{\|u\|_{\infty} }+1) ].
\end{equation}

\noindent Choosing these functions such that $\sup_{i} |t_{i+1}-t_{i}| $ and $\kappa$ tends to $0$ and placing them in equation \eqref{boundontraject} and then the equivalent of equation \eqref{ineqofcontrol}  with $\bar{x}$ replaced by $\widetilde{x}$ gives the result. \hfill$\square$\\

An important point is that this lemma applies to the value function $V_{\epsilon}$ as it can be interpreted as the value function associated with the Lagrangian
\begin{equation}
\bar{L}_{\epsilon}(t,x,u):=L\left(t,x,\frac{x}{\epsilon},u,\omega\right)
\end{equation}
which is uniformly continuous for each fixed $\epsilon$. Although its modulus of continuity deteriorates as $\epsilon \rightarrow 0$, this is not an issue. The effect of the deteriorating modulus of continuity is that the step functions will gain in  complexity. However, in the case where we use then for effective boundedness this does not become a problem. In the other case when we use it to prove Lemma \ref{The lemma that uses step functions}, the Lagrangian $\widetilde{L}$ is uniformly continuous.

In particular, what will we do in the proof of Lemmas \ref{effectiveboundednessofvaluelemma} and \ref{Boundedness lemma2}  is, for any control $u_{n} \in \mathcal{S}$, we will take the time intervals $[t_{i}, t_{i+1})$, to be so small that  $|u_{n}(s)-u_{n}(t)| \leq \lambda$ for all $t,s \in [t_{i}, t_{i+1})$ for each $i$, where the term $\lambda$ is as in Assumption \ref{Growthassumption}. Our ability to do this depends on the function $H$ and consequently will depend also on $\|u\|_{\infty}$.

\subsection{Approximation by bounded control}

\begin{lemma}
\label{bdd-control}
Let $[t_0,t_0+h]\subset [t,T]$ and $u\in\mathcal{U}$ be such that
\begin{align}
& f(x(s),u(s))=\sum_{i=0}^N v_{i} 1_{[t_{i}, t_{i+1})}(s),\\
& \max_{0\le j \le N}\sup_{s,r\in [t_j,t_{j+1}]}|u(s)-u(r)|\le \lambda,\label{osc}\\ 
& \int_{t_0}^{t_0+h} L\left(s,x(s), \frac{x(s)}{\epsilon}, u(s)\right)ds \le Wh \label{BOUND}
\end{align}
for some $W>0$, where $\lambda$ is as in Assumption \ref{Growthassumption}. Then for sufficiently large $R>0$ depending only on $W$, there exists a control $u_R\in \mathcal{U}^R$ such that 

\begin{equation}
  \int_{t_0}^{t_0+h}  L\left(s,x_R(s), \frac{x_R(s)}{\epsilon}, u_R(s)\right)ds \leq \int_{t_0}^{t_0+h} L\left(s,x(s), \frac{x(s)}{\epsilon}, u(s)\right)ds
\end{equation}
and $x(t_0+h)=x_R(t_0+h)$. 
\end{lemma}

\begin{proof}

We assume $t_0=0$ without loss of generality. The assumption $L_*(0)>-\infty$ implies that the Lagrangian $L$ is bounded from below. Since we are dealing with a finite time horizon problem, we assume $L\ge 0$ by adding a constant if necessary.

Let $[t_{j},t_{j+1})$ be the last time period such that $\max_{s\in[t_j,t_{j+1}]} |u(s)| \geq R $. If there is no such time period for $R$ sufficiently large, there is nothing to prove. Define $\hat{t}_j: = \text{argmin}_{s \in [t_j, t_{j+1}]} |u(s)|$. By \eqref{osc}, we have $\sup_{s\in[t_j,t_{j+1}]}|u(s)|\le |u(\hat{t}_j)|+\lambda$ and hence
\begin{equation}
|v_j|=|f(x(t_j),u(t_j))|
\le \frac{L_*(|u(\hat{t}_j)|)}{\gamma(|u(\hat{t}_j)|)}
\le \frac{L_*(|u(\hat{t}_j)|)}{\gamma(R-\lambda)}, 
\label{coercivity-v_j}
\end{equation}
where $\gamma$ is defined in Assumption \ref{Growthassumption}. 
We divide the proof into 7 steps. \\

\noindent\textbf{Step 1}: Fix an $N>0$ and set
\begin{equation}
\zeta_N(r):=| \{ s \leq r \colon |u(s)| \leq N \} |.
\end{equation}
If we take $N$ sufficiently large depending on $W$, we have $\zeta_N(h) \geq h/2$. 
\begin{proof}[Proof of Step 1]
Observe that
\begin{equation}
\int_{\{ s \leq r \colon |u(s)| > N \}}L_\epsilon(s; x(s), \frac{x(s)}{\epsilon}, u(s))ds
\ge L_*(N) (h-\zeta_N(h)). \label{step1-lower}
\end{equation}
On the other hand, since $L$ is assumed to be non-negative,  
\begin{equation}
\int_{\{ s \leq r \colon |u(s)| \leq N \}}L\left(s,x(s), \frac{x(s)}{\epsilon},u(s)\right)ds
 \le Wh.\label{step1-upper}
\end{equation}
Combining \eqref{step1-lower} and \eqref{step1-upper}, we arrive at
\begin{equation}
h-\zeta_N(h) \le \frac{Wh}{L_*(N)}  
\end{equation} 
and, since $\lim_{N\to\infty} L_*(N)=\infty$, we are done. 
\end{proof}

\noindent\textbf{Step 2}: We can make
\begin{equation}
\frac{1}{h}\int_{t_{j}}^{t_{j+1}} f(x(s),u(s))ds =(t_{j+1}-t_j)\frac{v_j}{h}
\end{equation}
as small as we wish by choosing $R$ large depending only on $W$. 
\begin{proof}[Proof of Step 2]
By the same way as in the proof of Step 1, we obtain
\begin{equation} 
\begin{split}
 (t_{j+1}-t_j)L_*(|u(\hat{t}_j)|)
 & \le \int_{t_j}^{t_{j+1}} L_{*}(|u(s)|) ds  
\le Wh.
\end{split}
\label{compareineq}
\end{equation}
Using this bound in \eqref{coercivity-v_j}, we obtain
\begin{equation}
(t_{j+1}-t_j)|v_j| \le \frac{Wh}{\gamma(R-\lambda)}
\end{equation}
and thanks to Assumption \ref{Growthassumption}, we are done.
\end{proof}

\noindent\textbf{Step 3}: In this step, we construct a time-change function. For $N<R-\lambda$ and $\beta>0$, define 
\begin{equation}
\rho(s)=\int_0^s \left(1_{r\not\in[t_j,t_{j+1})}-\beta 1_{u(r)<N}
 +\frac{|v_j|}{\delta}1_{r\in[t_j,t_{j+1})}\right)dr.
 \end{equation}
 
 If $s \geq t_{j+1}$ this is equal to
 
 \begin{equation}
 s-\beta\zeta_N(s)+(t_{j+1}-t_j)\frac{|v_j|}{\delta}.
\end{equation}

We choose 
\begin{equation}
\beta=(t_{j+1}-t_j)\frac{|v_j|}{\delta\zeta_N(h)}
\label{def-beta}
\end{equation}
so that $\rho(h)=h$. By Steps 1 and 2, we can make $\beta$ as small as we wish by choosing $R>0$ large depending only on $W$ and we always assume $\beta<1$. Our time-change $\sigma$ is defined as the inverse function of $\rho$. Note that $\rho(h)=h$ implies $\sigma(h)=h$. Noting also that $\min_{r\in[t_j,t_{j+1})}|u(r)|\ge R-\lambda \ge N$, we have 
\begin{equation}
\begin{split}
\frac{d\sigma}{ds}(s)&=\left(\frac{d\rho}{ds}(\sigma(s))\right)^{-1}\\
& = \frac{1_{\sigma(s)\not\in[t_j,t_{j+1})}}{1-\beta 1_{u(\sigma(s))<N}}
+\frac{\delta}{|v_j|}1_{\sigma(s)\in[t_j,t_{j+1})}.
\end{split}
\label{sigma'}
\end{equation}
The time-changed trajectory $x\circ\sigma$ has the velocity
\begin{equation}
\begin{split}
\frac{d x(\sigma(s))}{ds}
&=\frac{dx}{ds}(\sigma(s))\frac{d\sigma}{ds}(s)\\
&= f(x(\sigma(s)),u(\sigma(s)))\frac{d\sigma}{ds}(s),
\end{split}
\end{equation}
which means that we speed up by the factor $(1-\beta)^{-1}$ when 
$|u(\sigma(s))|<N$ and change the speed to $\delta$ when $\sigma(s)\in [t_j,t_{j+1})$.\\

\noindent\textbf{Step 4}: When $R>0$ is sufficiently large depending on $W$, there exists a control $\bar{u}\in \mathcal{U}$ which satisfies the following: 
\begin{enumerate}
 \item The associated trajectory is $x\circ\sigma$, that is, 
\begin{equation}
 f(x(\sigma(s)),\bar{u}(s))=f(x(\sigma(s)),u(\sigma(s)))\frac{d\sigma}{ds}(s), 
\label{changed-velocity}
\end{equation}
 \item when $\sigma(s)\in[t_j,t_{j+1})$, $|\bar{u}(s)|\le M$ ,
 \item when $|u(\sigma(s))|<N$, 
\begin{equation}
|u(\sigma(s))-\bar{u}(s)|\le \frac{\beta}{1-\beta} \|H\|_{Lip}^N f^*(N)
\end{equation}
and in particular $|\bar{u}(s)|\le R$.
\end{enumerate}
\begin{remark}\label{no-new-exceedance}
By the second condition, the control is made small on the last interval where it exceeded $R$. The third condition ensures that we have not created a new point where the control exceeds $R$, i.e., $|\bar{u}|>R$ only on $\bigcup_{i<j}\sigma^{-1}([t_i,t_{i+1}))$. 
\end{remark}

\begin{proof}[Proof of Step 4]
First, when $\sigma(s)\not\in[t_j,t_{j+1})$ and $|u(\sigma(s))|\ge N$, we have 
$\frac{d\sigma}{ds}(s)=1$ and we can take $\bar{u}(s)=u(\sigma(s))$. 

Second, when $\sigma(s)\in[t_j,t_{j+1})$, since the right hand side is on $\delta\mathbb{S}^d$, we can use \eqref{hereisdelta} to find $\bar{u}(s)\in U^M$ satisfying~\eqref{changed-velocity}. 
Since $M$ is the constant fixed in \eqref{hereisdelta}, we have $|\bar{u}(s)|\le R$ by choosing $R>M$. 

Finally, when $\sigma(s)\not\in[t_j,t_{j+1})$ and $|u(\sigma(s))|\le N$, the right hand side of~\eqref{changed-velocity} is a small perturbation of $f(x(\sigma(s)),u(\sigma(s)))$. Hence we can use Assumption~6 to find $\bar{u}(s)\in U$ satisfying~\eqref{changed-velocity}: 
specifically, the $R$ in Assumption~\ref{implicitfunctionassumption} is set to be $N$ and then we take our $R$ so large (i.e., $\beta$ small) that 
\begin{equation}
\frac{1}{1-\beta}f(x(\sigma(s)),u(\sigma(s)))\in B_{\eta(N)}(f(x(\sigma(s)),u(\sigma(s)))). 
\end{equation}
Then we can define the desired control as 
\begin{equation}
 \bar{u}(s)=H\left(x(\sigma(s)), \frac{1}{1-\beta}f(x(\sigma(s)),u(\sigma(s)))\right).
\end{equation}
Recalling the identity $H(x(\sigma(s)), f(x(\sigma(s)),u(\sigma(s))))=u(\sigma(s))$ and the Lipschitz continuity of $H$ in Assumption~\ref{implicitfunctionassumption}, we have 
\begin{equation}
\begin{split}
 |u(\sigma(s))-\bar{u}(s)|
&\le \|H\|_{Lip}^N\left(\frac{1}{1-\beta}-1\right)
 |f(x(\sigma(s)),u(\sigma(s)))|\\
& \le \|H\|_{Lip}^N\frac{\beta}{1-\beta} f^*(N).
\end{split}
\end{equation}
Taking $R$ large makes $\beta$ small and we can conclude $|\bar{u}(s)|\le R$. 
\end{proof}

\noindent\textbf{Step 5}: 
The following hold:
\begin{align}
 &|\{s\colon \sigma(s)\in [t_j,t_{j+1})\}|
 =(t_{j+1}-t_j)\frac{|v_j|}{\delta},\\
 &\sup_{s\in[0,h]}|\sigma(s)-s| 
 \le \frac{\beta}{1-\beta} h+(t_{j+1}-t_j)\frac{|v_j|}{\delta}. 
 \label{mod-sigma}
\end{align}
\begin{proof}[Proof of Step 5]
The first claim is a consequence of the fact that $x\circ\sigma$ travels the distance $(t_{j+1}-t_j)|v_j|$ in speed $\delta$. 
For the proof of the second claim, by \eqref{sigma'}, we know that $\frac{d\sigma}{ds}(s)\neq 1$ only if $u(\sigma(s))<N$ or $\sigma(s)\in [t_j,t_{j+1})$. Therefore, 
\begin{equation}
\begin{split}
|\sigma(s)-s| &\le \int_0^s \left|\frac{d\sigma}{ds}(s)-1\right|ds\\
&\le \left(\frac{1}{1-\beta}-1\right)\zeta_N(h)
 + \left|\frac{\delta}{|v_j|}-1\right||\{s\colon \sigma(s)\in [t_j,t_{j+1})\}|.
\end{split}
\label{sigma-s}
\end{equation}
Substituting the first claim and $\zeta_N(h)\le h$, we obtain~\eqref{mod-sigma}.
\end{proof}

\noindent\textbf{Step 6}: When $R>0$ is sufficiently large depending on $W$, 
\begin{equation}
\int_{0}^{h} L(s,x(s),u(s))ds 
\ge \int_{0}^{h} L(s,\bar{x}(s),\bar{u}(s))ds.
\end{equation}

\begin{proof}[Proof of Step 6]
We assume $\frac{\beta}{1-\beta}<2\beta$ by choosing $R>0$ large enough. 
The left hand side is bounded from below by
\begin{equation}
 (t_{j+1}-t_j) L_{*}(u(\hat{t}_j)) 
+ \int_{0}^{t_{j}} L(s,x(s),u(s)) ds
+ \int_{t_{j+1}}^{h} L(s,x(s),u(s)) ds. 
\label{LHS}
\end{equation}
Making the substitution $s=\sigma(r)$, recalling $x(\sigma(r))=\bar{x}(r)$ and noting that $\sigma'(r) \ge 1$ on the above two domains of integration, we can bound the second term from below by
\begin{equation} \label{hithere1}
 \int_{0}^{\sigma^{-1}(t_{j})}  L(\sigma(r),\bar{x}(r),u(\sigma(r)))\sigma'(r) dr
 \geq \int_{0}^{\sigma^{-1}(t_{j})} L(\sigma(r),\bar{x}(r),u(\sigma(r)))dr
\end{equation}
Applying the same argument to the third term in \eqref{LHS}, we find that the sum of two integrals in \eqref{LHS} is bounded from below by 
\begin{equation}
 \int_{[0,\sigma^{-1}(t_j))\cup[\sigma^{-1}(t_{j+1}),h]} L(\sigma(r),\bar{x}(r),u(\sigma(r)))dr.
\end{equation}
Using Steps 4 and 5, we can evaluate the error of replacing $\sigma(r)$ and $u(\sigma(r))$ by $r$ and $\bar{u}(r)$ in this integral and obtain
\begin{equation}
\begin{split}
&\int_{[0,\sigma^{-1}(t_j))\cup[\sigma^{-1}(t_{j+1}),h]} L(\sigma(r),\bar{x}(r),u(\sigma(r)))dr\\
&\quad \ge 
\int_{[0,\sigma^{-1}(t_j))\cup[\sigma^{-1}(t_{j+1}),h]} L(r,\bar{x}(r),\bar{u}(r))dr\\
&\qquad
-h\|L\|_{Lip}\left(2\beta h+(t_{j+1}-t_j)\frac{|v_j|}{\delta}\right)
-h\|L\|_{Lip}^N \times 2\beta \|H\|_{Lip}^N f^*(N)\\
&\quad =\int_{[0,\sigma^{-1}(t_j))\cup[\sigma^{-1}(t_{j+1}),h]} L(r,\bar{x}(r),\bar{u}(r))dr
-c(W)(t_{j+1}-t_j)\frac{|v_j|}{\delta},
\end{split} 
\end{equation}
where in the last line we have used the definition \eqref{def-beta} of $\beta$ and Step 1. (Recall that $N$ depends only on $W$.)
On the remaining interval $[\sigma^{-1}(t_j),\sigma^{-1}(t_{j+1}))$, we have $|\bar{u}|\le M$ by Step 4 and hence combining with Step 5,  
\begin{equation}
 \int_{[\sigma^{-1}(t_j),\sigma^{-1}(t_{j+1}))} L(r,\bar{x}(r),\bar{u}(r))dr
\le L^*(M)(t_{j+1}-t_j)\frac{|v_j|}{\delta}.
\end{equation}
Substituting all the above estimates to \eqref{LHS}, we find that 
\begin{equation}
\begin{split}
\int_{0}^{h} L(s,x(s),u(s))ds 
& \ge \int_0^h L(r,\bar{x}(r),\bar{u}(r))dr\\
&\quad
 + (t_{j+1}-t_j) \left(L_{*}(u(\hat{t}_j)) 
 -(c(W)+L^*(M))\frac{|v_j|}{\delta}\right).
\end{split}
\end{equation}
By \eqref{coercivity-v_j} and Assumption \ref{Growthassumption}, the second term on the right hand side is positive for sufficiently large $R>0$ depending only on $W$.
\end{proof}

\noindent\textbf{Step 7}: We can apply the above procedure to $\bar{u}$ again and continue recursively. As we noticed in Remark \ref{no-new-exceedance}, this procedure always decreases the number of intervals where the control exceeds $R$. Therefore, after finitely many steps, we end up with a new control $u_R\in \mathcal{U}^R$ that gives a smaller cost than the original one.
\end{proof}

\begin{remark}
As we proved in section $2.5$, the same estimates on the homogenized control problem as well as the discrete problems are applicable. Hence we can easily extend Lemma \ref{bdd-control} to control problems with Lagrangians $\widetilde{L}, L_{\tau,\epsilon}$. 
\end{remark}

\noindent Using Lemma \ref{bdd-control} we can now readily prove Lemmas \ref{effectiveboundednessofvaluelemma},  \ref{Boundedness lemma2} and \ref{effectiveboundednessofvaluelemma2}.\\

\begin{proof}[Proof of Lemma \ref{Boundedness lemma2}] For the object $\widehat{L}_{\tau, \epsilon}(t_{0},x_{0},\widetilde{u},\omega) $,  if we consider the control $u(s)$ such that $f(x(s),u(s))=f(x_{0},\widetilde{u})$,  then we conclude that

\begin{equation}
\widehat{L}_{\tau, \epsilon}(t_{0},x_{0},\widetilde{u},\omega)  \leq h L^{*}(|\widetilde{u}|+\eta(|\widetilde{u}|)).
\end{equation}
by taking $W= L^{*}(|\widetilde{u}|+\eta(|\widetilde{u}|))$. If $u_{n}$ is a sequence of controls such that

\begin{equation}
\lim_{n \rightarrow \infty} \int_{t_{0}}^{t_{0}+h} L(t_{0},x_{0}, \frac{x(s)}{\epsilon}, u_{n}(s))ds = \widehat{L}_{\tau, \epsilon}(t_{0},x_{0},\widetilde{u},\omega)\leq h L^{*}(|\widetilde{u}|+\eta(|\widetilde{u}|))
\end{equation}
which we can without loss of generality take $u_{n} \in \mathcal{S}$ by Lemma \ref{StepfunctionLemma}, then by Lemma \ref{bdd-control} we find a sequence of controls $\bar{u}_{n}(s)$ such that $\bar{u}_{n}(s) \in \mathcal{U}^{K_{1}}$ for some $K_{1}$ depending only on $L^{*}(|\widetilde{u}|+\eta(|\widetilde{u}|))$ and 
 
 \begin{equation}
 \int_{t_{0}}^{t_{0}+h} L(t_{0},x_{0}, \frac{x(s)}{\epsilon}, \bar{u}_{n}(s))ds \leq \int_{t_{0}}^{t_{0}+h} L(t_{0},x_{0}, \frac{x(s)}{\epsilon}, u_{n}(s))ds 
 \end{equation}
  and therefore we conclude
 
 \begin{equation}
 \widehat{L}_{\tau, \epsilon}(t_{0},x_{0},\widetilde{u},\omega) =  \widehat{L}^{K1}_{\tau, \epsilon}(t_{0},x_{0},\widetilde{u},\omega).
 \end{equation}
  \\
  It is trivial to repeat the proof of Lemma \ref{bdd-control} for the frozen dynamics $f(x_{0},\widetilde{u})$. Then, taking $W=L^{*}(|\widetilde{u}|)$ we can show for some $K_{2}$

 \begin{equation}
L_{\tau, \epsilon}(t_{0},x_{0},\widetilde{u},\omega)=  L^{K_{2}}_{\tau, \epsilon}(t_{0},x_{0},\widetilde{u},\omega)
 \end{equation}
 
\noindent  Taking $K=K_{1} \wedge K_{2}$ and observing that this $K$ depends only on $\widetilde{u}$ completes the proof. \end{proof}

\begin{proof}[Proof of Lemma \ref{effectiveboundednessofvaluelemma} and \ref{effectiveboundednessofvaluelemma2}]:  
Take a sequence of controls $u_{n} \in \mathcal{S}$ such that 

\begin{equation}
\lim_{n \rightarrow \infty} J_{\epsilon}(t,x,u_{n},\omega) = V_{\epsilon}(t,x).
\end{equation}

\noindent Clearly there exists an $M>0$ such that $J_{\epsilon}(t,x,u_{n},\omega) \leq M$ for all $n$. Therefore we can conclude that there exists a $K> 0$ and sequence of controls $\bar{u}_{n} \in \mathcal{U}^{K}$ such that

\begin{equation}
 J_{\epsilon}(t,x,\bar{u}_{n},\omega) \leq  J_{\epsilon}(t,x,u_{n},\omega)  
\end{equation}

\noindent in part owing to the fact that $x_{n}(T)= \bar{x}_{n}(T)$ and hence $\psi( x_{n}(T))=\psi( \bar{x}_{n}(T))$.  Thus

\begin{equation}
V^{K}(t,x,\omega) \leq \lim_{n \rightarrow \infty} J_{\epsilon}(t,x,\bar{u}_{n},\omega) = V(t,x,\omega)
\end{equation}

The reverse inequality is trivial and hence this completes the proof of Lemma \ref{effectiveboundednessofvaluelemma}. Adopting the same argument for $\widetilde{J}$ and $J_{\tau,\epsilon}$ we can immediately obtain Lemma \ref{effectiveboundednessofvaluelemma2}
\end{proof}

\section{Acknowledgments}

This research was supported by  a Monbukagakusho scholarship administered by the Japanese Ministry of Education, Culture, Sports, Science and Technology. I would also like to thank Ryoki Fukushima for his helpful insight and guidance.

\bibliography{Homogenization_of_some_determinstic_control_problem}

\end{document}